\documentclass[10pt]{amsart}

\usepackage{amssymb}
\usepackage{amsmath}
\allowdisplaybreaks
\usepackage{amsfonts}
\usepackage{fancybox}
\usepackage{empheq}
\usepackage{mathtools}
\usepackage{hyperref}

\usepackage[table]{xcolor}

\usepackage{color}
\usepackage{longtable,booktabs}
\usepackage{calrsfs} 
\usepackage{tikz-cd}
\usepackage[all,cmtip]{xy}

\usepackage{todonotes}
\usepackage{cleveref}

\definecolor{LightCyan}{rgb}{0.88,1,1}
\definecolor{LightCyan1}{rgb}{0.80,1,1}
\definecolor{Gray}{gray}{0.9}
\definecolor{Gray1}{gray}{0.95}

\usepackage{tikz}

\usepackage{changes}
\definechangesauthor[color=orange,name={Aristides Kontogeorgis}]{AK}

\newcommand\mycom[2]{\genfrac{}{}{0pt}{}{#1}{#2}}

\usepackage{bbm}

\usepackage{scalerel,stackengine}
\stackMath
\newcommand\reallywidehat[1]{%
\savestack{\tmpbox}{\stretchto{%
  \scaleto{%
    \scalerel*[\widthof{\ensuremath{#1}}]{\kern-.6pt\bigwedge\kern-.6pt}%
    {\rule[-\textheight/2]{1ex}{\textheight}}%WIDTH-LIMITED BIG WEDGE
  }{\textheight}% 
}{0.5ex}}%
\stackon[1pt]{#1}{\tmpbox}%
}
\newcommand{\Z}{\mathbb{Z}}

\newcommand{\Q}{\mathbb{Q}}

\newtheorem{theorem}{Theorem}
\newtheorem{lemma}[theorem]{Lemma}

\newtheorem{proposition}[theorem]{Proposition}
\theoremstyle{definition}
\newtheorem{example}[theorem]{Example}
\newtheorem{remark}[theorem]{Remark}
\newtheorem{definition}[theorem]{Definition}

\DeclareRobustCommand{\mybox}[2][gray!20]{%
\begin{tcolorbox}[   %% Adjust the following parameters at will.
        breakable,
        left=0pt,
        right=0pt,
        top=0pt,
        bottom=0pt,
        colback=#1,
        colframe=#1,
        width=\dimexpr\textwidth\relax, 
        enlarge left by=0mm,
        boxsep=5pt,
        arc=0pt,outer arc=0pt,
        ]
        #2
\end{tcolorbox}
}

 \usepackage[theorems,breakable]{tcolorbox}
\tcbuselibrary{theorems}
 \newtcbtheorem[number within=section]{impdef}{Definition}%
 {colback=yellow!5,colframe=LightCyan!,fonttitle=\bfseries,coltitle=black}{th}

\title[Generalized Fermat curves]{Galois action on homology of generalized Fermat Curves}

\date{\today}

\author[A. Kontogeorgis]{Aristides Kontogeorgis}
\address{Department of Mathematics, National and Kapodistrian  University of Athens
Pane\-pist\-imioupolis, 15784 Athens, Greece}
\email{kontogar@math.uoa.gr}

\author[P. Paramantzoglou]{Panagiotis Paramantzoglou }
\address{Department of Mathematics, National and Kapodistrian University of Athens
Pane\-pist\-imioupolis, 15784 Athens, Greece}
\email{pan\_par@math.uoa.gr}

\date \today

\makeatletter
\newcommand{\aprod}{\mathop{\operator@font \hbox{\Large$\ast$}}}
\makeatother

\begin{document}

\begin{abstract}
The fundamental group of Fermat and generalized Fermat curves is computed. 
These curves are Galois ramified covers of the projective line with abelian Galois groups $H$. 
We  provide a unified study of the action of both cover Galois group $H$ and  the absolute Galois group $\mathrm{Gal}(\bar{\Q}/\Q)$ on the pro-$\ell$ homology
of the curves in study.  
Also the relation to the pro-$\ell$ Burau representation is investigated. 
\end{abstract}

\maketitle

 % \tableofcontents

%
%
%
\section{Introduction}
%
%
%
% {\color{blue}
% \begin{itemize}
% 	\item It is very difficult to work with general curve
% 	\item Superelliptic curves provide us with concrete examples of curves to work with.
% 	\item Consider new infinite examples of coverings of curves. The Galois group is not part of the automorphisms if $g\geq 2$. 
% 	\item Families of curves filtered from bellow. 
% \end{itemize}
% }
% \highlight[id=AK,comment=3/3/2019]{
In \cite{MR4117575} we have studied the actions of the braid group and of the absolute Galois group on a cyclic cover of the projective line. 
In that article we have exploited the fact that in a ramified cover $X\rightarrow \mathbb{P}^1$ of the projective 
line 
we can remove the  ramified points and in this way  we obtain an open cover $X^0 \rightarrow X_s:=\mathbb{P}^1 \backslash \{P_1,\ldots,P_s\}$ of the projective line minus the ramified points.
 Fix a point $x_0 \in X_s$ and an arbitrary but fixed preimage $x_0'\in X^0$.
By covering space theory the open curve $X^0$ can be described as a quotient of the universal covering space $\tilde{X}_s$ by the fundamental group of the open curve $\pi_1(X^0,x_0')< \pi_1(X_s,x_0)  \cong F_{s-1}$, where $F_{s-1}$ is the free group in $s-1$ generators. Also the group $\pi_1(X^0,x_0')$ can be described as a subgroup of $F_{s-1}$ by using the Screier lemma technique, see \cite[sec. 3]{MR4117575}.   

Y. Ihara in \cite{IharaCruz}, \cite{Ihara1985-it} observed that 
 for a fixed prime $\ell$,
if we pass to the pro-$\ell$ completions of the fundamental groups of the above curves, then the  absolute Galois group $\mathrm{Gal}(\bar{\Q}/\Q)$ can be realized as a group of automorphisms of $\mathfrak{F}_{s-1}$ and it can also act on  certain subgroups $\mathfrak{F}_{s-1}$, corresponding to topological covers of $X_s$,
 here by $\mathfrak{F}_{s-1}$ we denote the free pro-$\ell$ group in $s-1$ generators. 
The fundamental group of $X_s$ admits the presentation
\[
\pi_1(X_s,x_0)\cong\langle
x_1,\ldots, x_s| x_1x_2\cdots x_s=1
\rangle.
\]
 In this way we can unify the study of both Braid group and the absolute Galois group on (co)homology spaces  of cyclic curves.  
In \cite[thm. 1]{MR4117575} the fundamental group $R_{\ell^k}$ of the open cyclic cover 
 $Y_{\ell^k}$ of the projective line minus $s$-points is computed to be equal to 
\begin{equation}
\label{Rlk}
R_{\ell^k}=
\langle
x_1^i x_j x_1^{-i+1}, 0\leq i \leq \ell^k-2, 2\leq j \leq s-1, x_1^{\ell^k-1}x_j, 1\leq j \leq s-1
\rangle.
\end{equation}
\begin{definition}
Given a group $F$ we will denote by $F'$ the derived subgroup of $F$, that is  $F'=[F,F]$. 
\end{definition}

In this article we continue this study by focusing on the case of certain abelian coverings of the projective line. 
We begin by studying the classical Fermat curves $\mathrm{Fer}_n$ given as projective algebraic curves by the equation 
\[
\mathrm{Fer}_n:=\{[x:y:z] \in \mathbb{P}^3: x^n+y^n+z^n=0\}. 
\]
The Fermat curve $\mathrm{Fer}_n$
forms a  ramified Galois 
cover of the projective line ramified over three points, $\{0,1,\infty\}$
with Galois group $H_0=\Z/n\Z \times \Z/n\Z$. 
More precisely the fundamental group of $X_3=\mathbb{P}^1\backslash \{0,1,\infty\}$ is isomorphic to the free group $F_2$ in two generators $a,b$. 
Here we may fix a point $x_0$ in $X_3$ and take as $a,b$ the homotopy classes of  loops circling once clockwise around the points $0,1$ of $\mathbb{P}^1$. There is also a third loop $c$ circling around $\infty$, but the homotopy class of this  loop can be expressed in terms of $a,b$ in terms of the relation $abc=1$.  
When we remove the preimages of the three ramified points $\{0,1,\infty\}$ from $\mathrm{Fer}_n$ we obtain the open Fermat curve $\mathrm{Fer}_n^0$, which is a topological covering of $X_3$. 
Let $R_{\mathrm{Fer}_n}=\pi_1(\mathrm{Fer}_n^0,x_0')$ be the  fundamental group of the open 
Fermat curve $\mathrm{Fer}_n^0$ and $x_0'\in \mathrm{Fer}_n^0$ is a fixed preimage of the point $x_0\in X_3$. The group $R_{\mathrm{Fer}_n}$ is known to be isomorphic to $\langle a^n,b^n, [a,b] \rangle < F_2$,  while 
\[
H_0= F_2/R_{\mathrm{Fer}_n}=F_2/ \langle a^n,b^n, [a,b] \rangle.
\]
Using this quotient we can define the following generators for the group $H_0$, namely $\alpha= a R_{\mathrm{Fer}_n}$ and $\beta=b R_{\mathrm{Fer}_n}$. 
Let $R_{\mathrm{Fer}_n}'$ denote the commutator group of the fundamental group 
$R_{\mathrm{Fer}_n}$
of $\mathrm{Fer}_n$. 
The group $F_2$ acts on $F_2$ by conjugation, that is for every two elements
$x,y\in F_2$ we define $x^y=yxy^{-1}$. 
Notice also that the homology group $R_{\mathrm{Fer}_n}/R_{\mathrm{Fer}_n}'$ becomes an $H_0$-module by defining 
\[
x^\alpha =a x a^{-1}, x^{\beta}= b x b^{-1} \text{ for } x \in R_{\mathrm{Fer}_n}/R_{\mathrm{Fer}_n}'.
\]
This action is well defined and independent of the selection of the representative of the class $\alpha,\beta \in H_0$. 

\begin{theorem}
% Let $\alpha, \beta$ be the generators of the group $H_0=\Z/n\Z \times \Z/n\Z$. 
% Also for an element $x\in F_{s-1}$ we will denote by 
The fundamental group $R_{\mathrm{Fer}_n}$ of the open Fermat curve $\mathrm{Fer}_n^0$
% seen as $H_0=\Z/n\Z \times \Z/n\Z$-cover of $\mathbb{P}^1-\{0,1,\infty\}$ 
is the  subgroup of the free group $F_2=\langle a, b\rangle$ on the generators
\mybox[gray!10]{
 \begin{align*}
A_1 &=  \left\{ (b^n)^{a^i}: 0 \leq i \leq n-1 \right\},
&  \#A_1 &= n
\\
A_2 & =  \left\{ [b^j,a]^{a^{i}}: 1 \leq j \leq n-1, 0 \leq i \leq n-2\right\}, & \#A_2 &= (n-1)^2
 \\
A_3 & =  \{a ^n [a^{-1},b^j]: 0\leq j \leq n-1 \},
&  \#A_3 &= n
\end{align*}
}
% Moreover  
% set $H_0=\Z/n\Z \times \Z/n\Z$,  where $\alpha,\beta$ are generators of the two cyclic components of the group $H_0$. 
The module $R_{\mathrm{Fer}_n}/R_{\mathrm{Fer}_n}'$ is generated as  a $\Z[H_0]$-module by the elements $a^n,b^n$ and $[a,b]$. 
 An isomorphic image of the module $R_{\mathrm{Fer}_n}/R_{\mathrm{Fer}_n}'$ fits in the small exact sequence 
\mybox[gray!10]{
\[
% R_{\mathrm{Fer}_n}/R_{\mathrm{Fer}_n}'=\Z[\langle \alpha \rangle] \bigoplus \Z[\langle \beta \rangle] \bigoplus \Z[H_0]/I,
0 \rightarrow 
\Z[\langle \alpha \rangle] 
\bigoplus \Z[\langle \beta \rangle]
\rightarrow 
R_{\mathrm{Fer}_n}/R_{\mathrm{Fer}_n}' 
\rightarrow
\Z[H_0]/I 
\rightarrow 
0,
\]
}
\noindent where $I$ is the ideal of $\Z[H_0]$ generated by $\sum_{i=0}^{n-1} \alpha^i, \sum_{i=0}^{n-1} \beta^i$, or equivalently 
\mybox[gray!10]{
\[
\Z[H_0]/I =  J_{\langle \alpha \rangle} \otimes  J_{\langle \beta \rangle},
\]
}
\noindent
where $J_{\langle \alpha \rangle}$ (resp.  $J_{\langle \alpha \rangle}$) denotes the coaugmentation module for the cyclic group $\langle \alpha \rangle$ (resp. $\langle \beta \rangle$), see section \ref{sec:2.2}.
Finally if  $\mathbb{F}$ is  a field which contains the $n$-different $n$-th roots of $1$, then 
\mybox[gray!10]{
\begin{equation}
\label{FermatDecoK}
H_1(\mathrm{Fer}_n,\Z)\otimes_\Z \mathbb{F}= 
\bigoplus_{
  \substack{i,j=1 \\ i+j\neq n} 
}
^ {  n-1 }
\mathbb{F} {\chi_{i,j}},
\end{equation}
}
\noindent where $\chi_{i,j}$ is the character of $H_0$ such that 
$\chi_{i,j}(\alpha^\nu,\beta^\mu)=\zeta_n^{i \nu+j \mu}$ 
 and $\zeta_n$ is a fixed primitive $n$-th root of unity. 
\end{theorem}
The proof of the above theorem can be found in section \ref{sec:2.2} and in particular in proposition \ref{prop:15}. 

The generalized Fermat curves play the role of Fermat curves in the more general setting of abelian coverings of $X_s=\mathbb{P}^1\backslash \{P_1,\ldots,P_s\}$, $s > 3$.  
 % that is the projective line minus more that three points removed. 
 Their automorphism group was recently studied by R. Hidalgo, M. Leyton-\'Alvarez and the authors in \cite{HiRuKoPa}.

% Let $\mathbbm{k}$ be a field of characteristic $p\neq \ell$ (in most applications it will be an algebraic 
% number field) or $\mathbb{C}$ and $\bar{\mathbbm{k}}$ be its algebraic closure.
A generalized Fermat curve of type $(k,s-1)$, where $k,s-1 \geq 2$ are integers, is a non-singular irreducible projective algebraic curve $ C_{k,s-1} $ defined over a field $\mathbbm{k}$   admitting a group of automorphisms
  $H_0 \cong ({\mathbb Z}/{k}\Z)^{s-1}$ so that $ C_{k,s-1} /H_0$ is the projective line with exactly $s$ branch points, each one with ramification index $k$. Such a group $H_0$ is called a generalized Fermat group of type $(k,s-1)$. 
Let us consider a branched regular covering $\pi: C_{k,s-1}  \to {\mathbb P}^{1}$, whose deck group is $H_0$. 
Let $R_{k,s-1}$ be the fundamental group of the open generalized Fermat curve $C_{k,s-1}^0=C_{k,s-1} \backslash \pi^{-1}(X_s)$ of type $(k,s-1)$.

By composing by a suitable M\"obius transformation (that is, an element of ${\rm PSL}_{2}(\mathbb{C})$) at the left of $\pi$, we may assume that the branch values of $\pi$ are given by the points
$\infty, 0,1, \lambda_{1}, \ldots, \lambda_{s-3},$
where $\lambda_i \in \mathbb{C} \backslash \{0,1\}$ are pairwise different, that is we can take $X_s=\mathbb{P}^1\backslash \{\infty,0,1, \lambda_1,\ldots,\lambda_{s-3}\}$, $s > 3$.

 A generalized Fermat curve of type $(k,s-1)$ can be seen as a complete intersection in a projective 
space $\mathbb{P}^{s-1}$, defined by the following set of equations
\begin{equation} \label{defineGFC}
C_{k,s-1}=C_{k}(\lambda_{1},\ldots,\lambda_{s-3}):=\left \{ \begin{array}{rcc}
              x_0^k+x_1^k+x_2^k&=&0\\
              \lambda_1x_0^k+x_1^k+x_3^k&=&0\\
              \vdots\hspace{1cm} &\vdots &\vdots\\
              \lambda_{s-3}x_0^k+x_1^k+x_{s-1}^k&=&0\\
             \end{array}\right \}\subset {\mathbb P}^{s-1}.
\end{equation} 
Observe that topologically the construction of generalized Fermat curves does not depend on the configuration of the ramification points. On the other hand the Riemann surface/algebraic curve structure depends heavily on this configuration. For instance the automorphism group depends on the configuration of these points, see \cite{HiRuKoPa}.

The genus of $ C_{k,s-1}$ can be  computed using the Riemann-Hurwitz formula: 
\begin{equation}\label{genero}
g_{(k,s-1)}=1+ \frac{k^{s-2}}{2} ((s-2)(k-1)-2).
\end{equation}
It is known \cite{Gonzalez-Diez2009-md} that generalized Fermat curves, 
have the orbifold uniformization $\mathbb{H}/\Gamma$ in terms of the Fuchsian group 
\begin{equation} \label{genFermatCurveFuchs}
  \Gamma_k=
  \langle
  x_1,x_2,\ldots,x_{s} \mid x_1^k=\cdots=x_{s}^k=x_1x_2\cdots x_s=1
  \rangle.
\end{equation}
The surface group is given, see \cite{Gonzalez-Diez2009-md}, \cite{Maclachlan} as $F_{s-1}\cdot \langle x_1^k,\ldots,x_{s-1}^k, (x_1\cdots x_{s-1})^k \rangle$. We will compute the genus of the generalized Fermat curves by two more different methods in eq. \ref{secondgenusFormula} and in section \ref{alexGFC}.

Recall that  $R_{k,s-1}$ denotes the fundamental group of the open generalized Fermat curve $C_{k,s-1}^0$ and set $\mathfrak{F}_{s-1,k}=\mathfrak{F}_{s-1}/\mathfrak{R}_k$, where $\mathfrak{R}_k$ is the smallest closed subgroup containing all elements $x_i^{\ell^k}, 0\leq i \leq s$.
\begin{theorem}
 The group $R_{k,s-1}$  is a free group generated by the union of the sets
 \mybox[gray!10]{
\begin{align}
\label{A-categories1}
A_{s-1} &=
\left\{
(x_{s-1}^k)^{x_1^{i_1}\cdots x_{s-2}^{i_{s-2}}}
\right\}   
\\ \nonumber
B_\nu
% &=
% \left\{
% x_{1,\nu-1}
% ^
% {
%   \mathbf{i}
% }
% \cdot x_\nu^{i_\nu} x_{\nu+1,s-1}
% ^
% {
%   \mathbf{i}
% }
%  \cdot x_\nu\cdot
%  \left(
% x_{\nu+1,s-1}
% ^
% {
%   \mathbf{i}
% }
% \right)
% ^{-1}\cdot x_\nu^{-i_\nu-1}\cdot 
% \left(
% x_{1,\nu-1}
% ^
% {
%   \mathbf{i}
% }
% \right)^{-1}
% \right\}
% \\\nonumber
% &= \left\{
% [x_{\nu+1,s-1}^{  \mathbf{i}
% }
% , x_\nu]^{x_{1,\nu-1}
% ^\mathbf{i}
% \cdot x_\nu^{i_\nu}}
% \right\}
% \\ \nonumber
&= 
\left\{
[x_{\nu+1,s-1}
^\mathbf{i}
, x_\nu]^{x_{1,\nu}^{
\mathbf{i}
}}
\right\}
\quad 1\leq i_\nu \leq k-2
\\
\nonumber
B_\nu'
% &=
% \left\{
% x_{1,\nu-1}^{  \mathbf{i}   }
% % x_1^{i_1}\cdots x_{\nu-1}^{i_{\nu-1}}
% \cdot
% x_\nu^{k-1}
% \cdot
% x_{\nu+1,s-1}^ \mathbf{i} 
% % x_{\nu+1}^{i_{\nu+1}} \cdots
% % x_{s-1}^{i_{s-1}}
% \cdot
% x_\nu
% \cdot
% \left(
% x_{\nu+1,s-1}^{  \mathbf{i}}
% \right)^{-1}
% \cdot
% \left(
% x_{1,\nu-1}^{  \mathbf{i}}
% \right)^{-1}
% \right\}
% \\ \nonumber
&=
\left\{
\left(
x_\nu^k [x_\nu^{-1},x_{\nu+1,s-1}^{  \mathbf{i}}]
\right)^{x_{1,\nu-1}^{  \mathbf{i}}}
\right\},
\end{align} 
}
\noindent where
$
x_{\ell_1,\ell_2}^{\mathbf{i}}
=x_{\ell_1}^{i_{\ell_1}} x_{\ell_1+1}^{i_{\ell_1+1}} 
\cdots x_{\ell_2}^{i_{\ell_2}}$,
$\mathbf{i}=(i_1,\ldots,i_{s-1})$,
 $0\leq i_j\leq k-1$, $1\leq j \leq s-1$.
The group $R_{k,s-1}/R_{k,s-1}'$ is also generated (not necessary freely) by the 
 union of the sets
\mybox[gray!10]{
  \begin{align}
\label{second-set-of-non-free-gens}
\tilde{A}_{s-1} &= \{
(x_{s-1}^k)^{x_{1, s-2}} \} 
\\ \nonumber
\tilde{A}_\nu &= \{ (x_\nu ^k)^{x_{1,\nu-1} \cdot x_{\nu+1,s-1}} \}, \text{ for } 1\leq \nu \leq s-2
\\ \nonumber
\tilde{A}_\nu' & =\{[x_j, x_\nu]^{x_{1,\nu-1} \cdot x_\nu^{i_\nu} \cdot x_{\nu+1,s-1}}\}, \text{ for } 1\leq \nu \leq s-2,
\end{align}}

\end{theorem}
\begin{proof}
This theorem is proved using the  Schreier lemma in section \ref{sec: Generalized Fermat Curves}. The transition from the first set of generators of eq. (\ref{A-categories1})  to the second set of non free generators of eq. (\ref{second-set-of-non-free-gens}) is done in proposition \ref{prop21}.
% Notice that the generators given above are not free generators, a set of free generators is given in lemma \ref{lemma7}.
\end{proof}
In our pro-$\ell$ setting we are interested in generalized Fermat curves of type $(\ell^k,s-1)$, so we will restrict ourselves to the study of curves $C_{\ell^k,s-1}$. 
Set
\(
\mathcal{I}:= ( \Z \cap [0,\ell^k))^{s-1}.
\)
Fix a primitive $\ell^k$-root of unity $\zeta_{\ell^k}$ and for each $\mathbf{i}=(i_1,\ldots,i_{s-1}) \in \mathcal{I}$ define the characters $\chi_{\mathbf{i}}(\cdot)$ on the abelian group 
\[
H_0= \{\mathbf{x}=(\bar{x}_1^{\nu_1},\ldots,\bar{x}_{s-1}^{\nu_{s-1}}): \bar{x}_{\mu}^{\nu_\mu} \in \mathbb{Z}/\ell^k \mathbb{Z} \} \cong (\mathbb{Z}/\ell^k \mathbb{Z})^{s-1}
\]
by 
\[
\chi_{\mathbf{i}}(\mathbf{x})=\zeta_{\ell^k}^{\sum_{\mu=1}^{s-1} \nu_\mu i_\mu}.
\]
\begin{theorem}
The pro-$\ell$ homology of the closed curve is given by 
\begin{equation}
\label{th4eq1}
H_1(C_{\ell^k,s-1},\Z)\otimes_\Z \Z_\ell=
% \left(
% \frac{
% \mathfrak{R}_{k,s-1} \cap \mathfrak{R}_k}{\mathfrak{R}_k}
% \right)^{\mathrm{ab}}=
\frac{\mathfrak{F}_{s-1,k}'}{\mathfrak{F}_{s-1,k}''}.
 \end{equation}
% where $\mathfrak{R}_{k,s-1}$ is the pro-$\ell$ completion of the group $R_{\ell^k,s-1}$,  and $\mathfrak{R}_{k,s-1}=\mathfrak{F}_{s-1}/\mathfrak{R}_k$.
Let  $\mathbb{F}$ be a field containing $\Z_\ell$ and the $\ell^k$-roots of unity. We have the following decomposition: 
\mybox[gray!10]{
\[
H_1( C_{\ell^k,s-1} ,
  \mathbb{F})=\bigoplus_{\mathbf{i}\in \mathcal{I}} \mathbb{F} \cdot
C(\mathbf{i}) \chi_{\mathbf{i}},
\]
}
\noindent where 
\[
C(\mathbf{i})=
\begin{cases}
s-z(\mathbf{i})-2 & \text{if } \mathbf{i}
\neq (0,\ldots,0)
\\
s-z(\mathbf{i}) & \text{if } \mathbf{i}=
(0,\ldots,0)
\end{cases}
\]
 and $z(\mathbf{i})$ is defined in eq. (\ref{zDef}). 
Moreover 
\[
\mathrm{rank}_{\Z_\ell} H_1( C_{\ell^k,s-1} ,\Z_\ell)=
(s-1)\left( \ell^k \right)^{s-1} +2 -s \left(\ell^k \right)^{s-2}.
\]
Let $\mathbbm{k} \subset \mathbb{C}$ and let  $\bar{\mathbbm{k}}$ be the  algebraic closure of $\mathbbm{k}$. The generalized Fermat curves behave in general in the same way  if $\mathbbm{k}$ is a field of characteristic $p\neq \ell$, but in this article we need to use fundamental groups and the theory of covering spaces so it is easier to assume that  $\mathbbm{k} \subset \mathbb{C}$, instead of working with algebraic fundamental groups in the sense of Grothendieck \cite{SGA1}, \cite{SGA1a}. In general there will be no difference if we work over a field of characteristic zero. 

Fix the number of ramified points $s$. If $\bar{\mathbbm{k}}(  C_{\ell^k,s-1} )$ is the function field of the generalized Fermat curve then 
\[
\frac{\mathfrak{F}_{s-1,k}'}{\mathfrak{F}_{s-1,k}''}=
\mathrm{Gal}
\big(
\bar{\mathbbm{k}} ( C_{\ell^k,s-1} )^{\mathrm{unrab}}/
\bar{\mathbbm{k}} ( C_{\ell^k,s-1} 
)
% K_k^{\mathrm{unrab}}/K_k
\big),
\]   
where $\bar{\mathbbm{k}}(  C_{\ell^k,s-1} )^{\mathrm{unrab}}$ is the maximal abelian unramified extension of the function field 
$\bar{\mathbbm{k}}(  C_{\ell^k,s-1} )$. 
% \highlight[id=AK,comment=5/3/2/19]{Tate module action of cover group on homology}
\end{theorem}
\begin{proof}
The group theoretic interpretation of homology given in eq. (\ref{th4eq1}) is proved  in section \ref{sec.5.1}. The analysis into characters is proved in section \ref{sec:Alexander} and in particular in proposition \ref{prop35}.
\end{proof}

% $\quad$ 
We would like to construct a ``curve'' 
% {\color{red} it is not a curve!}
 $C_s$ which is a Galois cover of the projective line ramified over the set of $s$-points with Galois group $\mathrm{Gal}(C_s/\mathbb{P}^1)=\Z_\ell^{s-1}$.
This ``curve'' can only be defined as  the limit case of the generalized Fermat curves $ C_{\ell^k,s-1} $.

\noindent
\begin{minipage}{0.45\textwidth}
 % as shown in the middle diagram.   Such a  definition needs some care. 
 We can avoid the definition of such a ``curve'' by working in the language of (infinite) Galois extensions of function fields as shown on the diagram on the right. 

  In this way instead of considering a simple generalized Fermat cover we consider all of them, together. 
\end{minipage}
\begin{minipage}{0.25\textwidth}
\[
\xymatrix{
  C_s \ar[dd]_{\Z_\ell^{s-1}} \ar[dr]^{(\ell^k \Z_\ell )^{s-1}} &  \\
  &  C_{\ell^k,s-1}  \ar[ld]^{(\Z/\ell^k \Z)^{s-1}} \\
  \mathbb{P}^1
}
\]
\end{minipage}
\begin{minipage}{0.25\textwidth}
\[
\xymatrix{
  M_s \ar@{-}[dd]_{\Z_\ell^{s-1}} \ar@{-}[dr]^{(\ell^k \Z_\ell )^{s-1}} &  \\
  & 
  \bar{\mathbbm{k}}(C_{\ell^k,s-1})  \ar@{-}[ld]^{(\Z/\ell^k \Z)^{s-1}} \\
  \bar{\mathbbm{k}}(t)
}
\]
\end{minipage}

The  pro-$\ell$ limit  
\[
\mathbb{T}:=\lim_{\mycom{\leftarrow}{\ell^k}} T(\mathrm{Jac}( C_{\ell^{k},s-1} ))=
\lim_{\mycom{\leftarrow}{ \ell^k}}
\frac{
  \mathfrak{F}_{s-1,k}'
}{
  \mathfrak{F}_{s-1,k}''
}
\]
corresponds to the $\Z_\ell$ homology of this  ``curve'' $C_s$ and all  the knowledge of the Galois module structure of all Tate modules of the curves $ C_{\ell^k,s-1} $ is equivalent to the knowledge of the Galois module stucture of $\mathbb{T}$. 

The situation is similar with the pro-$\ell$ Burau representation, defined in \cite{MR4117575}. We also in section \ref{Gassner2Burau} how we can pass from the $\Z_\ell^{s-1}$-covers corresponding to 
generalized Fermat curves, to the $\Z_\ell$-case corresponding to the pro-$\ell$ Burau representation, using the ideas of \cite{ParamPartC19}. 

Section \ref{sec:2} is devoted to the application of Schreier lemma to Fermat curves \ref{sec:FermatCurves} and generalized Fermat curves \ref{sec: Generalized Fermat Curves} and the computation of homology by passing to the abelianization of the fundamental group. 
Section \ref{sec:IharaRep} is an introduction to Ihara's ideas on the study of the absolute Galois groups as a profinite 
braid \cite{Ihara1985-it}, \cite{IharaCruz} following \cite{MorishitaATIT}. 
In section \ref{sec:Crowell} we compute the Alexander module for the generalized Fermat curves, while section \ref{sec:Magnus} is devoted to the $\Z_\ell^{s-1}$ cover of the projective line, seen as a limit of $ C_{\ell^k,s-1} $ 
curves and the relation to the Tate modules of them. Finally we consider the passage to the Burau representation by comparing the corresponding Crowell sequences, in terms of the viewpoint developed in \cite{ParamPartC19}.

 { \bf Acknowledgements:} We are indebted to  the anonymous referee for  his/her
thorough report and all 
remarks and corrections, which significantly contributed to improving our article.

\subsection{Geometric Interpretation}
\label{sec:GeometInterpret}
We consider a Galois covering $\pi:\bar{Y}\rightarrow \mathbb{P}^1$ of the projective line ramified above the points in $S\subset \mathbb{P}^1_{\bar{\Q}}$, and the corresponding covering of compact Riemann surfaces. 
We also assume that the genus $g$ of $\bar{Y}$ is $\geq 2$. 
 The curve $Y_0=\bar{Y} \backslash \pi^{-1}(S)$ is a topological covering of $X_s=\mathbb{P}^1_{\mathbb{C}} \backslash S$, which 
can be described in terms of covering theory and corresponds to a subgroup $R_0$ of 
\(
\pi_1(\mathbb{P}^1_{\mathbb{C}} \backslash S ) \cong 
\pi_1(\mathbb{P}^1_{\mathbb{\bar{\mathbb{Q}}}} \backslash S ) 
\rightarrow 
\pi_1^{\mathrm{pro}-\ell}(\mathbb{P}^1_{\mathbb{C}}).
\)
Denote by $\bar{R}_0$ the closure of $R_0$ in $\pi_1^{\mathrm{pro}-\ell}(\mathbb{P}^1_{\mathbb{C}})$.

We have seen in \cite{MR4117575} and we will see in section  \ref{sec:FermatCurves} how this group $R_0$ can be computed by using the Schreier lemma.
For an application of this method to cyclic covers of the projective line we refer to \cite{MR4117575}. 
In order to pass from the open curve to the corresponding closed Riemann surface we consider the quotient by the group  $\Gamma$, which is  the closure in the  subgroup of $\mathfrak{F}_{s-1}$ generated by the stabilizers of ramification points, that is 
\begin{equation}
\label{GammaDef}
\Gamma=
  \langle x_1^{e_1},\ldots,x_{s}^{e_s} \rangle,
\end{equation}
where $e_1,\ldots,e_s$ are the ramification indices of the ramification points of $\pi:\bar{Y}\rightarrow \mathbb{P}^1$. 
In this article  for some elements $g_1,\ldots,g_t$ in a certain group we will denote by $\langle g_1,\ldots,g_t \rangle$ the closed subgroup 
generated by the elements  $\{g_1,\ldots,g_t\}$.

Notice that if $e_1=e_2=\cdots=e_s$ then $\Gamma$ is the closure of the group $\Gamma_k$ 
defined in eq. (\ref{genFermatCurveFuchs}). Later, for $e_1=\ldots=e_s=\ell^k$ we will denote this group by $\mathfrak{R}_k$.
The group $R=R_0/R_0 \cap \Gamma$ corresponds to the closed curve $\bar{Y}$ as a quotient of the hyperbolic plane. 
This geometric situation can be expressed in terms of the short exact sequence of groups  where the map $\psi$ is the natural onto map defined by sending $a \Gamma\mapsto a \bar{R}_0 \cdot \Gamma$.
\begin{equation}
\label{short-def}
1 \rightarrow  R=
\frac{\bar{R}_0}{\Gamma \cap \bar{R}_0} 
\cong 
\frac{\bar{R}_0\cdot \Gamma}{\Gamma}
\rightarrow  
\frac{\mathfrak{F}_{s-1}}{\Gamma} 
\stackrel{\psi}{\longrightarrow} 
\frac{\mathfrak{F}_{s-1}}{\bar{R}_0\cdot \Gamma}
 \rightarrow 1. 
\end{equation}
% Instead  of the ring $\mathcal{A}$, we will consider the
% ring $\mathcal{A}^{\bar{R}_0,\Gamma}=\Z_\ell[[\frac{\mathfrak{F}_{s-1}}{\bar{R}_0 \cdot \Gamma}]]$, 
% % \todo{check it again!}
% see section \ref{sec:Crowell}. Notice also that the ring $\mathcal{A}^{\bar{R}_0,\Gamma}$ is commutative since we have assumed $\mathfrak{F}_{s-1}' \subset R_0$, see also lemma \ref{comm-action}.
% \end{remark}

In this article we focus on the study of Fermat and generalized Fermat curves. Namely, in sections \ref{sec:FermatCurves} and \ref{sec: Generalized Fermat Curves} we compute the fundamental group of the corresponding curves. We also treat the classical Fermat curves $s=3$ since this computation is elementary, while for the generalized Fermat curves $s\geq 3$ more advanced tools are needed, namely the usage of Alexander modules and the Crowell exact sequence. 

% \newpage

%
%
%
\section{Generalized Fermat Curves}
\label{sec:2}
\subsection{Fermat Curves}
\label{sec:FermatCurves}
These curves are  ramified curves over $\mathbb{P}^1 \backslash \{0,1,\infty\}$ with deck group $\mathbb{Z}/n\mathbb{Z} \times \mathbb{Z}/n\mathbb{Z}$.
We have $\pi_1(\mathbb{P}^1 \backslash \{0,1,\infty\},x_0)\cong F_2 \cong\langle a,b \rangle$.
This curve is a generalized Fermat curve $C_{n,s-1}$ for $s=3$.

\begin{definition}
The commutator $[a,b]$ of two elements $a,b$ in a group is defined as  $[a,b]=aba^{-1}b^{-1}$.
\end{definition}
\begin{lemma} \label{com-lem-ide}
For any two elements $x,y$ of a group and any positive integer $j$ we have
\begin{enumerate}
\item
$[x^j,y]=[x,y]^{x^{j-1}} \cdot [x,y]^{x^{j-2}} \cdots [x,y]^x \cdot [x,y]$ \label{com-ide1}
\item
$[x,y^j]=[x,y]\cdot [x,y]^y \cdot [x,y]^{y^2} \cdots [x,y]^{y^{j-1}}$.
\end{enumerate}
\end{lemma}
\begin{proof}
See \cite[0.1 p.1]{DDMS}.
\end{proof}

We will employ the Schreier lemma for describing the fundamental group of the Fermat curve  of level $n$, as explained in \cite[sec. 3]{MR4117575}.
More precisely a (right) Schreier Transversal of a subgroup $H$ of a free group $F_{s-1}=\langle x_1,\ldots,x_{s-1}\rangle$ with basis $X=\{x_1,\ldots,x_{s-1}\}$ is a set $T=\{t_1=1,\ldots, t_n\}$ of reduced words such that each right coset of $H$ in $F_{s-1}$ contains a unique word of $T$ called the representative of this class and all initial segments of these words also lie in $T$. For every $g\in F_{s-1}$ we will denote by $\bar{g}$ the element of $T$ with the property $H g = H \bar{g}$. Schreier's lemma, see \cite[lemma 6]{MR4117575} asserts that $H$ is freely generated by the elements $\gamma(t,x):=t x \overline{t x}^{-1}$, $t\in T$, $x\in X$ and $tx \not\in T$, $\gamma(t,x)\neq 1$.

A  Schreier transversal  $T$  for
 the subgroup $R_{\mathrm{Fer}_n} \subset F_2$ such that $F_2/R_{\mathrm{Fer}_n}\cong \Z/n\Z \times \Z/n\Z$  is given by $a^ib^j$, $0\leq i,j \leq n-1$. 
The fundamental group of the  Fermat curve is isomorphic to  $R_{\mathrm{Fer}_n}$.

\begin{lemma}
\label{characterRF}
The group $R_{\mathrm{Fer}_n}$ is characteristic,  that is every automorphism $\sigma\in \mathrm{Aut} F_2$ keeps $R_{\mathrm{Fer}_n}$ invariant.
\end{lemma}
\begin{proof}
The group $R_{\mathrm{Fer}_n} \subset F_2=\langle a,b \rangle$, can be generated by the elements $ a^n, b^n, [a,b]$.
The automorphism group of the free group $F_n$, and in particular of $F_2$, is generated by Nielsen transformations $n_i$ and $n_{ij}$ \cite[th. 1.5 p. 125]{bogoGrp} which are defined as follows: The automorphism $n_i$ sends a free generator $x_i \mapsto x_i^{-1}$ and leaves all other generators unchanged while the automorphism $n_{ij}$ sends $x_i \mapsto x_i x_j$ and leaves all other generators unchanged. It is evident from the relations of $R_{\mathrm{Fer}_n}$, see also lemma \ref{lemma7}, that  $n_i(R_{\mathrm{Fer}_n})=R_{\mathrm{Fer}_n}$  and $n_{ij}(R_{\mathrm{Fer}_n})=R_{\mathrm{Fer}_n}$.  
\end{proof}

We also compute: 
% {\color{red} (This seems different form Ihara's computation on page 65)}
\[
\overline{a^i b^j b}=
\begin{cases}
a^i b^{j+1} & \text{ if } j < n-1 \\
a^i & \text{ if } j =n-1
\end{cases}
\]
and
\[
\overline{a^i b^j a}=
\begin{cases}
a^{i+1} b^{j} & \text{ if } i < n-1 \\
b^j & \text{ if } i =n-1
\end{cases}
\]
Thus
\begin{eqnarray*}
a^ib^j b \left(\overline{a^i b^j b}\right)^{-1}
& = &
\begin{cases}
a^i b^j b b^{-j-1}  a^{-i}=1 & \text{ if } j < n-1 \\
a^{i} b^{n} a^{-i} & \text{ if } j=n-1
\end{cases}
\\
a^ib^j a \left(\overline{a^i b^j a}\right)^{-1}
&= &
\begin{cases}
a^i b^j a b^{-j}  a^{-i-1} & \text{ if } i < n-1, j\neq 0 \\
1 & \text{ if }  i<n-1,j=0,\\
a^{n-1} b^{j} a b^{-j} & \text{ if } i=n-1
\end{cases}
\end{eqnarray*}
Consider the generators $ \alpha=a R_{\mathrm{Fer}_n},\beta=b R_{\mathrm{Fer}_n}$ of the  group $\mathbb{Z}/n \mathbb{Z} \times \mathbb{Z}/n\mathbb{Z}$.
Observe that there is a well defined action of $\alpha$ (resp. $\beta$) on
$R_{\mathrm{Fer}_n}/R_{\mathrm{Fer}_n}'$ given by conjugation, i.e.
\[
x^{\alpha}=x^a=axa^{-1} \qquad x^{\beta}=x^b=bxb^{-1}
\]
for all $x\in R_{\mathrm{Fer}_n}/R_{\mathrm{Fer}_n}'$. Notice that this is indeed an action which
implies that
\[
(x^{\alpha})^\beta=x^{\alpha \beta}=x^{\beta\alpha}=(x^\beta)^\alpha
\]
i.e. the actions of $\alpha$ and $\beta$ commute.

\begin{lemma}
\label{lemma7}
 % Using the conjugation action of $\mathbb{Z}/n\mathbb{Z}=\langle \alpha \rangle$ we can express 
 The generators  of the free group $R_{\mathrm{Fer}_n}$ as union of the tree following sets:
\begin{align*}
A_1 &=  \left\{ (b^n)^{a^i}: 0 \leq i \leq n-1 \right\},
&  \#A_1 &= n
\\
A_2 & =  \left\{ [b^j,a]^{a^{i}}: 1 \leq j \leq n-1, 0 \leq i \leq n-2\right\} & \#A_2 &= (n-1)^2
 \\
A_3 & =  \{a ^n [a^{-1},b^j]: 0\leq j \leq n-1 \}
&  \#A_3 &= n
\end{align*}
\end{lemma}
\begin{proof}
This is a direct consequence of the Schreier lemma. Notice also that 
the above given sets together give rise to  $n^2+1$ generators as predicted by Schreier index formula. Indeed, we compute
$
\#A_1+ \#A_2 + \#A_3 = n+(n-1)^2+ n=n^2+1
% & =& n^2.
$.
\end{proof}

\begin{lemma}
Fix $0\leq i \leq n-2$. We will prove that the $\mathbb{Z}$-module generated by
the elements 
\[\Sigma_1(i):=\{[b^j,a]^{\alpha^i}, 1\leq j \leq n-1 \}
\]
is the same 
 as 
the $\mathbb{Z}$-module generated by the elements
\[
\Sigma_2(i):=\{[b,a]^{\alpha^i\beta^j}, 1\leq j \leq n-2\}.
\]
\end{lemma}

\begin{proof}
We will use additive notation here. 
By lemma 
 \ref{com-lem-ide}(\ref{com-ide1})  for
% \deleted[id=AK,comment=18/2/2019]{
% \begin{align*}
%   [b,a]^{\alpha^i} &= [b,a]^{\alpha^i} \\
%   [b^2,a]^{\alpha^i} &=[b,a]^{\beta\alpha^i} + [b,a]^{\alpha^i} \\
%   \cdots &= \cdots\\
%   [b^{n-1},a] &= [b,a]^{\beta^{n-2}\alpha^i}+\cdots+ [b,a]^{\beta \alpha^i}+[b,a]^{\alpha^i},
% \end{align*}
% }
% i.e. for 
 $1\leq j \leq n-1$ and $0\leq i \leq n-2$ we have 
\begin{equation}
\label{eqforceAug}
[b^j,a]^{\alpha^i}=[b,a]^{( \beta^{j-1}+\beta^{j-2}
 +\cdots +\beta+1 )\alpha^i}.
\end{equation}
Similarly to eq. (\ref{eqforceAug})
\begin{equation}
\label{eqforceAug1}
[a^j,b]^{\beta^i}=[a,b]^{( \alpha^{j-1}+\alpha^{j-2}
 +\cdots +\alpha+1 )\beta^i}.
\end{equation}

This proves that the elements of the set $\Sigma_1 (i)$ are transformed to the elements of the set $\Sigma_2(i)$ in terms of an invertible block matrix where each block  is  the invertible $(n-1)\times (n-1)$ matrix with entries in $\mathbb{Z}$:
\[
\begin{pmatrix}
1 & 0 & \cdots & 0 \\
1 & 1 & \ddots & \vdots \\
\vdots & \ddots & \ddots &0 \\
1 & \cdots &  1 & 1
\end{pmatrix}
\]
Therefore $\Sigma_1 (i)$ and $\Sigma_2 (i)$ generate the same $\mathbb{Z}$-module.
\end{proof}

Notice also that
\begin{align}
\label{power-a}
(a^n)^{\beta^j} &= b^j a^{n-1} b^{-j}a^{-n+1} \cdot a^{n-1} b^j a b^{-j}=[b^j,a^{n-1}]+\underbrace{a^{n-1}b^{j}ab^{-j}}_{\in A_3} 
\\ \nonumber
 &= [b^j,a]^{\alpha^{n-2}+\alpha^{n-3}+ \cdots +\alpha+1} +a^{n-1}b^{j}ab^{-j}.
\end{align}
Set $\Sigma_1= \cup_i \Sigma_1(i)$. 
The above computation shows that $(a^n)^{\beta^j}$ can be written as a $\mathbb{Z}$-linear combination of elements of $ \Sigma_1$ (which generate $A_2$) and $A_3$. Moreover
\[
a^{n}[a^{-1},b^j]=(a^n)^{\beta^j}-[b,a]^{
\left(\sum_{k=0}^{j-1} \beta^k \right)\left(\sum_{\lambda=0}^{n-2}
\alpha^\lambda \right).
}
\]
We have shown that 
\begin{lemma}
\label{homoRF}
The free $\mathbb{Z}$-module $R_{\mathrm{Fer}_n}/R_{\mathrm{Fer}_n}'$ can be generated by the $n^2+1$ elements
\[
(a^n)^{\beta^i}, (b^n)^{\alpha^i}, 0\leq i \leq n-1 \text{ and }
[a,b]^{\alpha^i \cdot\beta^j}, 0\leq i,j \leq n-2.
\]
\end{lemma}
% \highlight[id=AK,comment=18/2/2019]
%
%
\subsection{Structure as a ${\Z/n\Z \times \Z/n\Z}$-module}
\label{sec:2.2}

We can now consider the homology group as the rank $n^2+1$ free $\mathbb{Z}$-module $R_{\mathrm{Fer}_n}/R_{\mathrm{Fer}_n}'$.
By lemma \ref{characterRF}
 $R_{\mathrm{Fer}_n}$ is a characteristic subgroup, so 
the group  $ H_0=\mathbb{Z}/n\mathbb{Z} \times \mathbb{Z}/n\mathbb{Z}=\langle \alpha \rangle \times \langle \beta \rangle$ acts on $R_{\mathrm{Fer}_n}/R_{\mathrm{Fer}_n}'$ by conjugation making $R_{\mathrm{Fer}_n}/R_{\mathrm{Fer}_n}'$ a $ H_0$-module.

For a finite group $G$ the coaugmenation ideal $J_G$ is defined
 as the quotient $J_G=\Z[G]/\langle \sum_{g\in G} g\rangle$.

\begin{lemma}
Set $H_0=\Z/n\Z \times \Z/n\Z$.
The module $R_{\mathrm{Fer}_n}/R_{\mathrm{Fer}_n}'$ is generated 
 as a $\Z[H_0]$-module by the elements $a^n,b^n,[a,b]$.  
An isomorphic image of the module $R_{\mathrm{Fer}_n}/R_{\mathrm{Fer}_n}'$ fits in the short exact 
sequence
\[
0 \rightarrow 
\Z[\langle \alpha \rangle] 
\bigoplus \Z[\langle \beta \rangle]
\rightarrow 
R_{\mathrm{Fer}_n}/R_{\mathrm{Fer}_n}' 
\rightarrow
\Z[H_0]/I 
\rightarrow 
0,
\]
% \[
% {\color{red} \cong}
% \Z[\langle \alpha \rangle] \bigoplus \Z[\langle \beta \rangle] \bigoplus \Z[H_0]/I,
% \]
where $I$ is the ideal of $\Z[H_0]$ generated by $\sum_{i=0}^{n-1} \alpha^i, \sum_{i=0}^{n-1} \beta^i$, or equivalently 
\[
\Z[H_0]/I \cong J_{\langle \alpha \rangle} \otimes J_{\langle \beta \rangle}.
\]
% \highlight[id=AK,comment=17/3/2019]{Coaugmentation instead of augmentation}
\end{lemma}
\begin{proof}
By the  $\Z$-basis given in lemma \ref{homoRF}
it is evident that $a^n,b^n,[a,b]$ indeed generate $R_{\mathrm{Fer}_n}/R_{\mathrm{Fer}_n}'$. 
The elements $a^n,b^n$ are acted by the groups $\langle \alpha \rangle$, $\langle \beta \rangle$ and form a $\mathbb{Z}[H_0]$-submodule of $R_{\mathrm{Fer}_n}/R_{\mathrm{Fer}_n}'$ isomorphic to $\Z[\langle \alpha \rangle] 
\bigoplus \Z[\langle \beta \rangle]$. Indeed, since the action of $\alpha$ on $a$ is trivial we can identify the elements in $\Z[H_0] a^n$ to the set of elements of $\Z[\langle \beta \rangle]$ and $\Z[H_0] b^n$ can be similarly identified to $\Z[\langle \alpha \rangle]$. Notice also that $\Z[H_0]a^n \cap \Z[H_0] b^n=\{0\}$.

Observe now that the elements $[a,b]^{\alpha^i \beta^j}$ are subject to the condition given in eq. (\ref{eqforceAug}) which implies that for all $i$
\begin{equation}
\label{beta-quot-zero}
[a,b]^{\alpha^i (1+\beta+\beta^2+\cdots 
 +
\beta^{n-1})}=[a,b^n]^{\alpha^i}= a^{i+1} b^n a^{-i-1} -a^ib^n a^{-i} \in \Z[\langle \alpha \rangle] 
\bigoplus \Z[\langle \beta \rangle].
\end{equation}
In the above formula we have used the additive structure of $R_{\mathrm{Fer}_n}/R_{\mathrm{Fer}_n}'$. 
Equation (\ref{beta-quot-zero}) shows that 
the operator $1+\beta+\cdots+\beta^{n-1}$ in the quotient of $R_{\mathrm{Fer}_n}/R_{\mathrm{Fer}_n}'$ by 
$\Z[\langle \alpha \rangle] 
\bigoplus \Z[\langle \beta \rangle]$
is zero. 
 Similarly, by  eq. (\ref{eqforceAug1}) we obtain that  $\sum_{i=0}^{n-1} \alpha^i=0$ in the quotient.
Therefore, in the $\Z[H_0]$-module generated by
 $[a,b]$ we have that $1+\beta+\cdots+\beta^{n-1}$ and $1+\alpha+\cdots+\alpha^{n-1}$ both annihilate $[a,b]$.
We compute that  
\[
\Z[H_0]/\langle 1+\beta+\beta^2+\cdots 
 + 
\beta^{n-1} \rangle
= \bigoplus_{i=0}^{n-1} \alpha^i J_{\langle \beta \rangle}.
\]
The result follows.  
\end{proof}
\begin{remark}
The theorem of Maschke implies that after a scalar extension to $\Q$ we have 
\[
R_{\mathrm{Fer}_n}/R_{\mathrm{Fer}_n}' \otimes_\Z \Q
 \cong
\Q[\langle \alpha \rangle] \bigoplus \Q[\langle \beta \rangle] \bigoplus \Q[H_0]/I,
\]
\end{remark}

We will now prove that in $R_{\mathrm{Fer}_n}/R_{\mathrm{Fer}_n}'$ there are exacty $3n$ elements which are fixed by an element of $\mathbb{Z}/n\mathbb{Z} \times \mathbb{Z}/n\mathbb{Z}$.

First note, the $2n$ elements
% The $2n$ elements are
 $(a^n)^{\beta^i}$ (respectively  $(b^n)^{\alpha^i}$)  are fixed by $\langle \alpha \rangle$ (respectively 
   $\langle \beta \rangle$).

The other $n$-elements are the elements $((ab)^n)^{\alpha^i}$ which are fixed by $\langle ab \rangle$. 
\begin{lemma}
We can write the elements $((ab)^n)^{\alpha^i}$ as follows:
\[
(ab)^n = [b,a]^{
\alpha^{n-1} \left(\sum_{\nu=0}^{n-2} \beta^\nu \right)
+ \cdots+
\alpha^2(\beta+1) +
\alpha
}
a^n  b^n.
\]
\end{lemma}
\begin{proof}
We begin by computing the $n=2$ case:
\begin{align*}
(ab)^2=abab=ab 
\overbrace{
ab^{-1}a^{-1}
}^{[b,a]}
abb=a[b,a]a^{-1}a abb=[b,a]^\alpha a^2b^2.
\end{align*}
Similarly, for the $n=3$ case we have
\begin{align*}
(ab)^3 
&=
(ab)^2 (ab)=[b,a]^\alpha a^2 b^2 ab
=
[b,a]^\alpha
a^2 b
\overbrace{
b ab^{-1}a^{-1}}^{[b,a]}
ab b\\
&= 
[b,a]^\alpha \cdot [b,a]^{\alpha^2 \beta} a^2b ab^2=
[b,a]^{\alpha+\alpha^2 \beta } a^2  
\overbrace{ba b^{-1} a^{-1}}^{[b,a]}
ab b^2 
\\
&=[b,a]^{\alpha+\alpha^2 \beta}[b,a]^{\alpha^2} a^3b^3=[b,a]^{\alpha^2(\beta+1)+\alpha} a^3b^3.
\end{align*}
Now assume that for $k$ we have 
\[
(ab)^k=[b,a]^{
\alpha^{k-1}
\left(
\sum_{\nu=0}^{k-2} \beta^\nu 
\right)
+ \cdots +
\alpha^2(\beta+1)+\alpha
}
a^k b^k
\] 
Set $E(k)=\alpha^{k-1}
\left(
\sum_{\nu=0}^{k-2} \beta^\nu 
\right)
+ \cdots +
\alpha^2(\beta+1)+\alpha$.
We will now consider
\begin{align*}
(ab)^{k+1} &=
[b,a]^{
E(k)
}
a^k b^k
ab
\\
&=
[b,a]^{E(k)}
a^k [b^k,a] ab^k b= [b,a]^{E(k)} [b^k,a]^{\alpha^k} a^{k+1} b^{k+1}
\\
&=[b,a]^{E(k)}[b,a]^{
\alpha^{k}
\left(
\beta^{k-1}+\beta^{k-2}+\cdots+ \beta+1
\right)
}
a^{k+1} b^{k+1}
\\
&=
[b,a]^{
  \alpha^{k}
  \left(
  \sum_{\nu=0}^{k-1} \beta^\nu
\right)
  + 
\alpha^{k-1}
\left(
\sum_{\nu=0}^{k-2} \beta^\nu 
\right)
+ \cdots +
\alpha^2(\beta+1)+\alpha
}
a^{k+1} b^{k+1}, 
\end{align*}
as desired.
\end{proof}
The above lemma gives us 
\[
((ab)^n)^{\alpha^i} = [b,a]^{
\alpha^{n-1+i} \left(\sum_{\nu=0}^{n-2} \beta^\nu \right)
+ \cdots+
\alpha^{2+i}(\beta+1) +
\alpha^{1+i}
}a^n \left( b^n \right)^{ \alpha^i }.
\]

We can see that the trasformation matrix from elements $[b^j,a]^{\alpha^i}$ to elements of the form $[b,a]^{\beta^j\alpha^i}$ is invertible. 
This allows us to prove that the  elements in the sets $A_2$ and $A_3$ can be written as linear combinations of elements of the form $[b,a]^{\alpha^i\beta^j}$ and 
$(a^n)^{\beta^j}$ for $1\leq j \leq n-1$, $0\leq i \leq n-2$.  It is clear that the elements $(a^n)^{\beta^i}, (b^n)^{\alpha^i}, ((ab)^n)^{\alpha^i}$ as given in the table below are fixed by the cyclic group mentioned in the third 
 column. 
The elements $\gamma_i$ are the $n$-elements $(b^n)^{\alpha^i}$ fixed by $\beta$, the $n$-elements $(a^n)^{\beta^i}$
fixed by $\alpha$ and the $n$ invariant elements 
$((ab)^n)^{\alpha^i}$
in the module generated by commutators. In the following table we enumerate the fixed elements $\gamma_i$:

\begin{longtable}[c]{c  l  c}
\toprule
Invariant element $\gamma_i$ &  Index & Fixed by\\
\tabularnewline
\midrule
\endhead
$(a^n)^{\beta^i}$ &  $1\leq i \leq n$ & $ \langle \alpha \rangle$\\
$(b^n)^{\alpha^i}$ &  $n+1\leq i \leq 2n$ & $\langle \beta \rangle$\\
$((ab)^n)^{\alpha^i}$ & $2n+1\leq i \leq 3n$ & $\langle \alpha \beta\rangle$ \\
\bottomrule
\end{longtable}
\bigskip
So far we have computed the open Fermat curve 
 admitting a presentation
\[
R_{\mathrm{Fer}_n}=\langle  a_1,b_1,\ldots,a_g,b_g,\gamma_1,\ldots,\gamma_{3n} | \gamma_1 \gamma_2 \cdots \gamma_{3n} \cdot[a_1,b_1][a_2,b_2]\cdots [a_g,b_g]=1\rangle,
\]
where $g$ is the genus of the closed Fermat curve which equals to $(n-1)(n-2)/2$. Every ramification point of the Fermat curve is surrounded by a path $\gamma_i$ and there are $3n$ such paths. 
We can verify that our computation 
 is correct
so far, by computing the genus of the closed Fermat curve. 
% For the computation of the genus of the closed curve, i.e. 
We add  the  $3n$ missing points and  we observe that the rank of $R_{\mathrm{Fer}_n}$ equals $2g+3n-1$ so the Schreier index formula implies:
\[
2g+3n-1=n^2+1 \Rightarrow g=\frac{n^2+2-3n}{2}=\frac{(n-1)(n-2)}{2}.
\]
\begin{figure}
\includegraphics[scale=0.5]{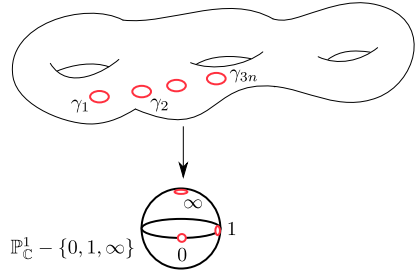}
\caption{Open Fermat curve as cover of the projective line}
\end{figure}

We have that
\[
H_1(X_F,\mathbb{Z})=\frac{R_{\mathrm{Fer}_n}/R_{\mathrm{Fer}_n}'}{\langle \gamma_1,\ldots,\gamma_{3n} \rangle}.
\]
% \newpage
\begin{definition}
We will denote by  $\zeta_n$ a fixed primitive $n$-th root of unity and by 
$\chi_{i,j}$  the character such that $\chi_{i,j}(\alpha^\nu\beta^\mu)=\zeta_n^{i\nu+j\mu}.$
\end{definition}

\begin{definition}
Let $\Gamma $ be  free $\mathbb{Z}$-module generated by $\langle \gamma_1,\ldots,\gamma_{3n}\rangle$.
\end{definition}

\begin{proposition}
\label{prop:15}
A basis for the $\mathbb{Z}$-module $H_1(X_F,\mathbb{Z})$ consists of the set:
\[
\{[b,a]^{\alpha^i \beta^j} \mod \Gamma: 0\leq i \leq n-2,0\leq j \leq n-3 \}.
\]
Let $\mathbb{F}$ be a field that contains $n$ different $n$-th roots of $1$.  Then 
\[
H_1(X_F,\Z) \otimes_\Z \mathbb{F}= 
\bigoplus_{
\substack{
i,j=1
\\
i+j \neq n}
}^{n-1} \mathbb{F} \chi_{i,j}.
\] 
\end{proposition}
\begin{proof}
The first assertion follows by considering the action modulo the elements which are invariant by an element of 
 $H_0$.
 Indeed, in  order to compute the quotient we change the basis of $R_{\mathrm{Fer}_n}/R_{\mathrm{Fer}_n}'$
by replacing each one of  the elements $[b,a]^{\alpha^{n-1+i}\beta^{n-2}}$ by $((ab)^n)^{\alpha^i}$ for all $0\leq i \leq n-1$, which belongs to the group $\langle \gamma_1,\ldots,\gamma_{3n} \rangle$ and is considered to be zero.

For the second assertion let us write 
\[
\left(
J_{\langle \alpha \rangle } \otimes \mathbb{F}
\right)
 \bigotimes 
\left(
J_{\langle \beta\rangle} \otimes \mathbb{F}
\right)
=
\left(
\bigoplus_{i=1}^{n-1} \mathbb{F} \chi_{i,0}
\right)
\bigotimes
\left( 
\bigoplus_{j=1}^{n-1} \mathbb{F} \chi_{0,j}
\right)
=\bigoplus_{i,j=1}^{n-1} \mathbb{F}\chi_{i,j}
\]
We are looking for the elements which are stabilized by $\alpha \beta$, that is $\chi_{i,j}(\alpha \beta)=\zeta^{i+j}=1$. This is the module $\oplus_{i=1}^{n-1} \mathbb{F} \chi_{i,n-i}$, which has $n$-elements.  The desired result follows. 
\end{proof}
Observe that the above computation agrees with $\dim_\mathbb{F} H_1(X_F,\Z) \otimes_\Z \mathbb{F}=(n-1)(n-2)$.
% \newpage

\subsubsection{Braid group action}
We will now consider the action of the Braid group $B_3$ on $H_1(X_F,\Z)$ of the closed Fermat surface.
By the faithful Artin representation we observe  that the braid group in three strings is generated by the elements
$\sigma_1 ,\sigma_2$, 
where
\[
\sigma_1(a) = a b a^{-1} \qquad  \sigma_2(a) = a 
\qquad
\sigma_1(b)  = a  \qquad  \sigma_2(b) =  a^{-1}b^{-1}.
\]
% \highlight[id=AK,comment=18/2/2019 $B_{s-1} \rightarrow B_s$]{Above we have used the extended action given in remark \ref{larger-action}.}
% \begin{alignat*}{4}
% \sigma_1(a) & = a b a^{-1} &\qquad && \sigma_2(a) &= a \\
% \sigma_1(b) & = a  &\qquad && \sigma_2(b) &=  a^{-1}b^{-1}
% \end{alignat*}
Notice that the above two automorphism in the abelianized free group with two generators acts like the matrices
\[
\bar{\sigma}_1=\begin{pmatrix}
0 & 1 \\ 1 & 0
\end{pmatrix}
\qquad
\bar{\sigma}_2=
\begin{pmatrix}
1 & -1 \\
0 & -1
\end{pmatrix},
\]
in $\mathrm{GL}(2,\mathbb{Z})$, reflecting the fact that $B_3/Z(B_3)\cong \mathrm{PSL}(2,\mathbb{Z})$.
Therefore, 
working in $R_{\mathrm{Fer}_n}/R_{\mathrm{Fer}_n}'$ we have
\begin{alignat*}{4}
\sigma_1[a,b]  &= [aba^{-1},a] &&= [b,a]^{\alpha}=-[a,b]^{\alpha}
\\
\sigma_2[a,b]  &= [a,a^{-1}b^{-1}] &&= [b^{-1},a^{-1}].
\end{alignat*}
and more generally
\[
\sigma_1([b,a]^{\alpha^i \beta^j}) =-[b,a]^{\alpha^{j+1}\beta^i} 
\qquad
\sigma_2([b,a]^{\alpha^i \beta^j}) =-[b^{-1},a^{-1}]^{\alpha^{i-j} \beta^{-j}
}
\]
Indeed, 
we compute
\begin{align*}
\sigma_1([b,a]^{\alpha^i \beta^j})
&=
\sigma_1
\left(
a^i b^j [b,a] b^{-j} a^{-i}
\right) =
ab^ia^{-1} a^j [a,b]^\alpha  a^{-j+1} b^{-i} a^{-1}
\\
&= -[b,a]^{\alpha^{j+1} \beta^i }
\end{align*}
and 
\begin{align*}
\sigma_2([b,a]^{\alpha^i \beta^j}) &=
\sigma_2
\left(
a^i b^j [b,a] b^{-j} a^{-i}
\right) =-
a^i (a^{-1} b^{-1})^j [b^{-1},a^{-1}] (a^{-1} b^{-1})^{-j} a^{-i}
\\
&=-[b^{-1},a^{-1}]^{\alpha^i (\alpha \beta)^{-j}}=
-[b^{-1},a^{-1}]^{\alpha^{i-j} \beta^{-j}}.
\end{align*}
In the above equations we have used that $H_0$ is abelian and its action on $R_{\mathrm{Fer}_n}/R_{\mathrm{Fer}_n}'$ is well defined. 
We also compute 
\begin{alignat*}{4}
\sigma_1( (b^n)^{\alpha^i} ) & =(a^n)^{\beta^i}  &&
\sigma_1( (a^n)^{\beta^j}) &&= (b^n)^{\alpha^{j+1}} \\
\sigma_2((b^n)^{\alpha^i}  )&= \left( (ba)^n \right)^{-\alpha^i} \qquad  &&
\sigma_2( (a^n)^{\beta^j}) &&= (a^n)^{(\beta \alpha)^{-j}}.
\end{alignat*}

% \begin{proposition}
% The group $H_1(X,\mathbb{Z})$ has the following $B_3$-module structure:
% \[
% {\color{red} aaa }
% \]
% \end{proposition}

% Automorphism group of the Fermat curve as ellements in the mapping class group. 

% {\color{red}
% Action of the algebra $\mathcal{A}=\mathbb{Z}_\ell[[u,v,w]]/\langle (u+1)(v+1)(w+1)-1\rangle\cong \mathbb{Z}_\ell[[u,v]]$.

% Comparison with the work of Ihara. Ihara wants to study the representation to $\mathcal{F'}/\mathfrak{F}''$ (Gassner?). He approaches it with Fermat curves, with fundamental groups $\mathfrak{F}' \mathfrak{R}_n$.
% (Is this the open or the closed curve?)
% }

%
\subsection{The Generalized Fermat Curve}
\label{sec: Generalized Fermat Curves}
% \todo{references and History, Lang, Rohrlich}

\subsubsection{Application of the Schreier lemma}
$\;$

Consider the open curve $X_s=\mathbb{P}^1 \backslash \{0,1,\infty,\lambda_1,\ldots,\lambda_{s-3}\}$  and let $x_0$ be a fixed base point in $X_s$. 

\noindent
\begin{minipage}{0.64\textwidth}
 For the  fundamental group we have
\[
\pi_1(X_s,x_o) \cong F_{s-1}=\langle x_1,\ldots, x_{s-1}\rangle.
\]
 Let $\tilde{X}_s$ denote the universal covering space and
$Y=\tilde{X}_s/F'_{s-1}$ be the cover of $X_s$ corresponding to the group $F_{s-1}'$, 
that is 
\[
\mathrm{Gal}(Y/X_s) \cong F_{s-1}/F_{s-1}'\cong H_1(X_s,\Z) \cong \Z^{s-1}.
\] 
Let $H_{k,s-1}\cong (\Z/k \Z)^{s-1}$ be the abelian group fitting in the short exact sequence
\begin{equation}
\label{multi-index}
0 \rightarrow I 
\rightarrow H_1(X_s,\Z) \rightarrow H_{k,s-1} \rightarrow 0. 
\end{equation}
% We have the picture on the right.
\end{minipage}
\begin{minipage}{0.35\textwidth}
\[
\xymatrix{
  \tilde{X}_s \ar@/_1pc/[ddd]^{F_{s-1}} \ar@/^0.8pc/[drr]^{F_{s-1}'} \ar[ddr]^{R_{k,s-1}} &  & \\
   & &  Y \ar@/^2pc/[ddll]^{H_1(X,\mathbb{Z})} \ar[dl]_{I} \\
    &  C_{k,s-1}  \ar[dl]_{H_{k,s-1}} & \\
   X_s &
}
\]
\end{minipage}
% \smallskip

If we denote, in additive notation, $H_1(X_s,\Z)=\bigoplus_{\nu=0}^{s-1} x_\nu \Z$, then $I=\bigoplus_{\nu=0}^{s-1} k  x_\nu \Z$.
 
\begin{remark}
The short exact sequence is in some sense a  generalization of the winding number  exact sequence given in definition 9 of \cite{MR4117575}.  
\end{remark}

We will now employ the Schreier lemma \cite[chap. 2 sec. 8]{bogoGrp}, \cite[sec. 2.3 th. 2.7]{MagKarSol} in order to compute the free subgroup $R_{k,s-1}\subset F_{s-1}$, 
 where $R_{k,s-1}$ is the subgroup of $F_{s-1}$ corresponding to the curve $C_{k,s-1}$ and is isomorphic to the fundamental group of $C_{k,s-1}$. 
We will introduce some new notation first: 
\begin{definition}
For any $\mathbf{i} := (i_1,\ldots, i_{s-1}) \in  \Z^{s-1}$, define
\[
\mathbf{x}^{\mathbf{i}}:= x_1^{i_1} x_2^{i_2} \cdots x_{s-1}^{i_{s-1}}.
\]
We also set  
\[
\mathbf{e}_1 = (1,0,\ldots,0), \mathbf{e}_2 = (0,1,\ldots,0),
\ldots,
\mathbf{e}_{s-1}=(0,\ldots,0,1).
\]
\end{definition}
% {\color{red}
% Recall that a (right) Schreier Transversal for a subgroup $R< F_{s-1}$
% is a set $T=\{t_1=1,\ldots,t_n\}$ of reduced words, such that each right coset of $R$ in $F_{s-1}$ contains a unique word of $T$ (called a representative of this class) and all initial segments of these words also lie in $T$. In particular $1$ lies in $T$ and represents the class $R$ and $Rt_i \neq R t_j$ for all $i\neq j$. For any $g\in F_{s-1}$ denote by $\overline{g}$ the unique element of $T$ with the property 
% $Rg=R\overline{g}$.
%  }
\begin{lemma}
A Schreier Transversal for  $R_{k,s-1}< F_{s-1}$ is given by 
\[
T=
\{\mathbf{x}^{\mathbf{i}}: \mathbf{i}=(i_1,\ldots,i_j,\ldots,i_{s-1}) \in \Z^{s-1} \text{ and } 0 \leq i_j \leq k-1\},
% \left\{
% x_1^{i_1}x_2^{i_2}\cdots x_{s-1}^{i_{s-1}}, 0\leq i_j \leq k-1, 1\leq j \leq s-1
% \right\}.
\]
\end{lemma}
\begin{proof}
Notice, that the set $T$ contains $\#H_{k,s-1}$ elements which are different modulo $R_{k,s-1}$. To see this we can use the fact that
\[
H_{k,s-1} \cong \frac{F_{s-1}}{R_{k,s-1}}=\frac{H_1(X_s,\Z)}{I},
\]
and the special form of $I$. The condition concerning the initial segments is trivially satisfied by the special form of the elements in $T$.
\end{proof}
For given $1 \leq \nu \leq s-1$ we have:
\[
\overline{
    \mathbf{x}^{\mathbf{i}} \cdot x_\nu
}
=
\begin{cases}
\mathbf{x}^{\mathbf{i}+ \mathbf{e}_\nu} & i_\nu < k-1 \\
\mathbf{x}^{\mathbf{i}-i_\nu \mathbf{e}_\nu} & i_\nu =k-1 
\end{cases}
\]

% \[
% \overline{
%   x_1^{i_1}x_2^{i_2}\cdots x_{s-1}^{i_{s-1}}\cdot x_\nu
% }=
% \begin{cases}
% x_1^{i_1}x_2^{i_2}\cdots x_\nu^{i_\nu+1}\cdots  x_{s-1}^{i_{s-1}} & \text{ if } i_\nu< k-1 \\
% x_1^{i_1}x_2^{i_2}\cdots x_{\nu-1}^{i_{\nu-1}}x_{\nu+1}^{i_{\nu+1}} \cdots x_{s-1}^{i_{s-1}} & \text{ if } i_\nu=k-1
% \end{cases}
% \]
% Denote by $\bar{x}^{\bar{i}}=x_1^{i_1}x_2^{i_2}\cdots x_{s-1}^{i_{s-1}}$. We now compute 

{\bf Case I} For $1\leq \nu \leq s-1$:

\[
\mathbf{x}^{\mathbf{i}} \cdot x_\nu \cdot
\left(
\overline{
\mathbf{x}^{\mathbf{i}} \cdot x_\nu
}
\right)^{-1}=
\begin{cases}
\mathbf{x}^{\mathbf{i}}\cdot x_\nu \cdot \mathbf{x}^{-\mathbf{i}-\mathbf{e}_\nu}
& \text{ if } i_\nu < k-1
\\
\mathbf{x}^{\mathbf{i}}\cdot x_\nu \cdot \mathbf{x}^{-\mathbf{i}+(k-1)\mathbf{e}_\nu} & \text{ if } i_\nu=k-1 
\end{cases}
\]
Notice that in the second case
\[
\mathbf{x}^{-\mathbf{i}+(k-1)\mathbf{e}_\nu}=
x_1^{-i_1}x_2^{-i_2}\cdots x_{\nu-1}^{-i_{\nu-1}}x_{\nu+1}^{-i_{\nu+1}} \cdots x_{s-1}^{-i_{s-1}}.
\]

% \[
% \bar{x}^{\bar{i}} \cdot x_\nu \cdot 
% \left(
% \overline{
%   \bar{x}^{\bar{i}} \cdot x_\nu
% }
% \right)^{-1}
% =
% \begin{cases}
% \bar{x}^{\bar{i}} \cdot x_\nu
% \cdot x_1^{-i_1}x_2^{-i_2}\cdots x_{\nu}^{-i_\nu-1} \cdots x_{s-1}^{-i_{s-1}}
% & \text{ if } i_\nu < k-1 
% \\
% \bar{x}^{\bar{i}} \cdot x_\nu \cdot 
% x_1^{-i_1}x_2^{-i_2}\cdots x_{\nu-1}^{-i_{\nu-1}}x_{\nu+1}^{-i_{\nu+1}} \cdots x_{s-1}^{-i_{s-1}}
% & \text{ if } i_\nu=k-1.
% \end{cases}
% \]

{\bf Case II} For  $\nu=s-1$:
\[
\mathbf{x}^i \cdot x_{s-1} \cdot 
\left( 
\overline{
\mathbf{x}^{\mathbf{i}} x_{s-1}
}
\right)^{-1} =
\begin{cases}
1 & i_{s-1} < k-1 
\\
\mathbf{x}^{\mathbf{i}} \cdot x_{s-1}^k \cdot (\mathbf{x}^{\mathbf{i}})^{-1} 
&
i_{s-1}=k-1
\end{cases}
\]
% \[
% \bar{x}^{\bar{i}} \cdot x_{s-1} \cdot 
% \left(
% \overline{
%   \bar{x}^{\bar{i}} \cdot x_{s-1}
% }
% \right)^{-1}
% =\]
% \[
% =\begin{cases}
% 1 & \text{ if } i_{s-1} < k-1
% \\
% \left(
% x_1^{i_1}x_2^{i_2}\cdots x_{s-2}^{i_{s-2}}
% \right)
%  \cdot x_{s-1}^k \cdot 
%  \left(
% x_1^{i_1}x_2^{i_2}\cdots x_{s-2}^{i_{s-2}}
% \right)^{-1}
% & \text{ if } i_{s-1}=k-1.
% \end{cases}
% \]
\[
x_{\ell_1,\ell_2}
^\mathbf{i}
=
\begin{cases}
x_{\ell_1}^{i_{\ell_1}} x_{\ell_1+1}^{i_{\ell_1+1}} 
\cdots x_{\ell_2}^{i_{\ell_2}}
& \text{ if } \ell_1 \leq \ell_2
\\
1
& \text{ if } \ell_1 > \ell_2
\end{cases}
\]
The generators of the free group $R_{k,s-1}$ are falling in the following categories: 
\begin{align}
\label{A-categories}
A_{s-1} &=
\left\{
(x_{s-1}^k)^{x_1^{i_1}\cdots x_{s-2}^{i_{s-2}}}
\right\}   
\\ \nonumber
B_\nu
&=
\left\{
x_{1,\nu-1}
^
{
  \mathbf{i}
}
\cdot x_\nu^{i_\nu} x_{\nu+1,s-1}
^
{
  \mathbf{i}
}
 \cdot x_\nu\cdot
 \left(
x_{\nu+1,s-1}
^
{
  \mathbf{i}
}
\right)
^{-1}\cdot x_\nu^{-i_\nu-1}\cdot 
\left(
x_{1,\nu-1}
^
{
  \mathbf{i}
}
\right)^{-1}
\right\}
\\\nonumber
&= \left\{
[x_{\nu+1,s-1}^{  \mathbf{i}
}
, x_\nu]^{x_{1,\nu-1}
^\mathbf{i}
\cdot x_\nu^{i_\nu}}
\right\}
\\ \nonumber
&= 
\left\{
[x_{\nu+1,s-1}
^\mathbf{i}
, x_\nu]^{x_{1,\nu}^{
\mathbf{i}
}}
\right\}
\quad 1\leq i_\nu \leq k-2
\\
\nonumber
B_\nu'
&=
\left\{
x_{1,\nu-1}^{  \mathbf{i}   }
% x_1^{i_1}\cdots x_{\nu-1}^{i_{\nu-1}}
\cdot
x_\nu^{k-1}
\cdot
x_{\nu+1,s-1}^ \mathbf{i} 
% x_{\nu+1}^{i_{\nu+1}} \cdots
% x_{s-1}^{i_{s-1}}
\cdot
x_\nu
\cdot
\left(
x_{\nu+1,s-1}^{  \mathbf{i}}
\right)^{-1}
\cdot
\left(
x_{1,\nu-1}^{  \mathbf{i}}
\right)^{-1}
\right\}
\\ \nonumber
&=
\left\{
\left(
x_\nu^k [x_\nu^{-1},x_{\nu+1,s-1}^{  \mathbf{i}}]
\right)^{x_{1,\nu-1}^{  \mathbf{i}}}
\right\}.
\end{align} 
Notice that $A_{s-1}$ corresponds to Case II, while the sets $B_\nu$, $B'_\nu$ for $1\leq \nu \leq s-2$ correspond to Case I, for $i_\nu<k-1$ and $i_\nu=k-1$ subcases respectively. 
We now count the sizes of the above sets. 
\begin{align*}
\# A_{s-1} &= k^{s-2} \\
\# B_\nu &= (k-1)\cdot k^{\nu-1} \cdot (k^{s-1-\nu}-1), \text{ for } 1\leq \nu \leq s-2 \\
\# B_\nu' &= k^{\nu-1}\cdot k^{s-1-\nu}=k^{s-2}, \text{ for } 1\leq \nu \leq s-2
\end{align*}
which gives in total 
\begin{equation}
\label{genuscount}
\# A_{s-1}+ \sum_{\nu=1}^{s-2} \#B_\nu + \sum_{\nu=1}^{s-2} \# B_\nu'= (s-2)\cdot k^{s-1}+1. 
\end{equation}
\subsubsection{Elements stabilized}
Let us not consider the action of 
 $H_{k,s-1}=(F_{s-1}/R_{k,s-1})$ on $(R_{k,s-1}/R'_{k,s-1} )$. Let us now denote 
\[
\frac{F_{s-1}}{R_{k,s-1}}=
\langle
\xi_1,\ldots, \xi_{s-1}
\rangle \cong \left( \frac{\Z}{k \Z} \right)^{s-1},
\]
where $\xi_j=x_j R_{k,s-1}$.
Let us write $\xi_{\ell_1,\ell_2}^{\mathbf{i}}= x_{\ell_1,\ell_2}^{\mathbf{i}}R_{k,s-1}$.
Observe first that the group generated by $\xi_{s-1}$ stabilizes
 $(x_{s-1}^k)^{\xi_{1,s-2}
^
{
  \mathbf{i}
}
 }$, since 
\begin{align*}
\left(
(
x_{s-1}^k
)^{\xi_{1,s-2}^{\mathbf{i}}}
\right)^{\xi_{s-1}} &=
(x_{s-1}^k)^{\xi_{s-1} \cdot \xi_{1,s-2}^{\mathbf{i}}} 
\\
&=
(x_{s-1} x_{s-1}^k x_{s-1}^{-1})^{\xi_{1,s-2}^{\mathbf{i}}}= (x_{s-1}^k)^{\xi_{1,s-2}^{\mathbf{i}}}. 
\end{align*}
  In this way we see that all $k^{s-2}$ elements of $A_{s-1}$ have non trivial stabilizer. 
Now we observe that 
\[
x_{1,\nu-1}^{  \mathbf{i}}
\cdot x_\nu^{k-1}\cdot 
x_{\nu+1,s-1}^{  \mathbf{i}}
\cdot x_\nu 
\cdot 
\left(
x_{\nu+1,s-1}^{  \mathbf{i}}
\right)^{-1}\cdot 
\left(
x_{1,\nu-1} ^{  \mathbf{i}}
\right)^{-1}=\]
\[=
[x_\nu^{k-1},x_{\nu+1,s-1}^{  \mathbf{i}}
]^{\xi_{1,\nu-1} ^{  \mathbf{i}}}
\cdot (x_\nu^k)^{\xi_{1,\nu-1}^{  \mathbf{i}}
 \cdot
\xi_{\nu+1,s-1}
^{  \mathbf{i}}
}.
\]
Observe that 
 for each $\nu$, $1\leq \nu \leq s-1$,  $\langle \xi_\nu\rangle$ stabilizes the $k^{s-2}$ elements of $B_\nu'$ of the form
  $(x_\nu^k)^{\xi_{1,\nu-1}^{  \mathbf{i}}
 \cdot
\xi_{\nu+1,s-1}^{  \mathbf{i}}}$ and the element $\langle \xi_1\cdots \xi_{s-1}\rangle$ stabilizes all elements  $((x_1\cdots x_{s-1})^k)^{\xi_{1,s-2}^{  \mathbf{i}}}$, 
 of which there are  $k^{s-2}$. 

\begin{longtable}[c]{c  l  l}
\toprule
Invariant element $\gamma_i$ &  Cardinal & Fixed by\\
\tabularnewline
\midrule
\endhead
$(x_{s-1}^k)^{\xi_{1,s-2}^{  \mathbf{i}}}$ &  $k^{s-2}$ & $ \langle \xi_{s-1} \rangle$
\\
$(x_\nu^k)^{\xi_{1,\nu-1}^{  \mathbf{i}} \cdot
\xi_{\nu+1,s-1}^{  \mathbf{i}}}$ 
&  $(s-2)k^{s-2}$ 
& $\langle \xi_\nu \rangle$, where  $1\leq  \nu \leq s-2$
 \\
$((x_1\cdots x_{s-1})^k)^{\xi_{1,s-2}^{  \mathbf{i}}}$ & $k^{s-2}$ & $\langle \xi_1\cdots \xi_{s-1}\rangle$ \\
\bottomrule
\end{longtable}
In total we have $s k^{s-2}$ fixed elements $\gamma_i$.

\begin{lemma}
\label{lem19}
The following equality holds. 
\begin{align*}
[x_{\nu+1,s-1}^{\mathbf{i}}, x_\nu]^{\xi_{1,\nu}^{\mathbf{i}}}
&= 
\left(
 [x_{\nu+1}^{i_{\nu+1}}, x_\nu]
+ 
[x_{\nu+2}^{i_{\nu+2}},x_\nu]^{{\xi_{\nu+1}}^{i_{\nu+1}}}+ 
\cdots
+[x_{s-1}^{i_{s-1}}, x_\nu]^{\xi_{\nu+1, s-2}^{\mathbf{i}} } 
\right)^{\xi_{1, \nu}^{\mathbf{i}}}.
\\
&= 
\left(
\sum_{j=\nu+1}^{s-1} [ x_j^{i_j},x_\nu]
^{\xi^{\mathbf{i}}_{\nu+1,j-1}    }
\right)^{\xi_{1, \nu}^{\mathbf{i}}}
\end{align*}
% \begin{multline*}
% [x_{\nu+1,s-1}, x_\nu]^{\xi_{1,\nu-1}\cdot \xi_\nu^{i_\nu}}
% = \\
% \left(
%  [x_{\nu+1}^{i_{\nu+1}}, x_\nu]^{{\xi_\nu}^{i_\nu}}
% + 
% [x_{\nu+2}^{i_{\nu+2}},x_\nu]^{{\xi_\nu}^{i_\nu} \cdot {\xi_{\nu+1}}^{i_{\nu+1}}}+ 
% \cdots
% +[x_{s-1}^{i_{s-1}}, x_\nu]^{\xi_{\nu, s-2}} 
% \right)^{\xi_{1, \nu-1}^{\mathbf{i}}}.
% \end{multline*}
\end{lemma}
\begin{proof}
The lemma will be proved by induction. Notice that it is enough to prove
\begin{align*}
[x_{\nu+1,s-1}^{\mathbf{i}}, x_\nu]
&= 
 [x_{\nu+1}^{i_{\nu+1}}, x_\nu]
+ 
[x_{\nu+2}^{i_{\nu+2}},x_\nu]^{{\xi_{\nu+1}}^{i_{\nu+1}}}+ 
\cdots
+[x_{s-1}^{i_{s-1}}, x_\nu]^{\xi_{\nu+1, s-2}^{\mathbf{i}} }
\\
&=
\sum_{j=\nu+1}^{s-1} [ x_j^{i_j},x_\nu]
^{\xi^{\mathbf{i}}_{\nu+1,j-1}    }.
\end{align*}
For $\nu+1=s-1$ the desired equality 
is trivial.
We will use the following commutator identity, which can be easily verified:
\[
[xz,y]=[z,y]^x \cdot [x,y].
\]
Assume that the equality holds  for the next product $x_{\nu,s-1}^{\mathbf{i}}$ we compute
\[
[x_{\nu,s-1}^{\mathbf{i}},x_\nu]=
[x_{\nu}^{i_\nu} \cdot x_{\nu+1,s-1}^{\mathbf{i}},x_\nu]=
[x_{\nu+1,s-1}^{\mathbf{i}},x_\nu]^{\xi_\nu^{i_\nu}} \cdot [x_\nu^{i_\nu},x_\nu]. 
\]
Hence writing the above equality  additively we obtain 
\begin{align*}
[x_{\nu,s-1}^{\mathbf{i}},x_\nu]
&=
[x_\nu^{i_\nu},x_\nu] + 
\left( 
[x_{\nu+1}^{i_{\nu+1}}, x_\nu]
+ 
[x_{\nu+2}^{i_{\nu+2}},x_\nu]^{\xi_{\nu+1}^{i_{\nu+1}}}+ 
\cdots
+[x_{s-1}^{i_{s-1}}, x_\nu]^{\xi_{\nu+1, s-2}^{\mathbf{i}} }
\right)^{\xi_\nu^{i_\nu}}
\\
&=
[x_\nu^{i_\nu},x_\nu]
+
[x_{\nu+1}^{i_{\nu+1}}, x_\nu]^{\xi_\nu^{i_\nu}}
+ 
[x_{\nu+2}^{i_{\nu+2}},x_\nu]^{\xi_\nu^{i_\nu}\xi_{\nu+1}^{i_{\nu+1}}}+ 
\cdots
+[x_{s-1}^{i_{s-1}}, x_\nu]^{\xi_{\nu, s-2}^{\mathbf{i}}  }
\end{align*}
\end{proof}

Similarly to the computation of the classical Fermat curves we change to a more suitable basis.
\begin{proposition}
\label{prop21}
Recall that $\mathbf{i}=(i_1,\ldots, i_{s-1})$,  $0 \leq i_1,\ldots,i_{s-1} \leq k-1$.
A generating set for $R_{k,s-1}/R_{k,s-1}'$ is given by 
% We can also take a new basis insteed of the old one. This is the following:
\begin{align*}
\tilde{A}_{s-1} &= \{
(x_{s-1}^k)^{\xi_{1, s-2}^\mathbf{i}   } \} 
\\
\tilde{A}_\nu &= \{ (x_\nu ^k)^{\xi_{1,\nu-1}^{\mathbf{i}} 
\cdot \xi_{\nu+1,s-1}^\mathbf{i}} \}, \text{ for } 1\leq \nu \leq s-2
\\
\tilde{A}_\nu' & =\{[x_j, x_\nu]^{\xi_{1,s-1}^{\mathbf{i}} 
% \cdot \xi_\nu^{i_\nu} 
% \cdot 
% \xi_{\nu+1,s-1}^\mathbf{i}
}\}, \text{ for } 0\leq j<\nu \leq s-1, 
1\leq i_j,i_\nu \leq k-2.
\end{align*}
\end{proposition}
\begin{proof}
Notice that eq. (\ref{power-a}) for $i\neq \nu$ implies that 
\begin{align}
\label{power-xk}
\left(
x_i^k
\right)^{\xi_\nu^j} &=[x_\nu^j,x_i^{k-1}]+x_{i}^{k-1}x_\nu^j x_i x_\nu^{-j}
\\ \nonumber
&= [x_\nu^j,x_i]^{\xi_i^{k-2} + \xi_i^{k-3}+\cdots+ \xi_i+1} + x_{i}^{k}
[x_i^{-1},x_\nu^j].
\end{align}
The result follows using the generating set given in eq. (\ref{A-categories}), lemma \ref{lem19} and eq. (\ref{power-xk}).
% Use relations to show that these are generators. Then count that they have the correct cardinality. 
% {\color{red}

% We have $\#A_{s-1}=\#A_\nu=k^{s-2}$ so we totally have $(s-1)k^{s-2}$ generators. 

% The elements in this page are $\binom{s-1}{2}=\frac{(s-1)(s-2)}2$
% }
\end{proof}

% \begin{proposition}
% An alternative basis of the free group is generated by the elements:
% {\small
% \begin{alignat*}{3}
% A_{s-1} &=\{ (x_{s-1}^k)^{x_{1,s-2}}\} && \qquad k^{s-2} \text{ elements } \\
% A_\nu' &= \{ (x_\nu^k)^{x_{1,\nu-1}\cdot x_{\nu+1,s-2}}\} && \qquad k^{s-2} \text{ elements } 1\leq \nu \leq s-2 \\
% A_\nu &= \{[x_j,x_\nu]^{x_{1,\nu-1}\cdot x_\nu^{i_\nu} \cdot x_{\nu+1,s-1}}\}
% && \qquad k^{\nu-1}(k-1)(k^{s-1-\nu}-1) \text{ elements}, 1 \leq \nu \leq s-2
% \end{alignat*}
% }
% where 
% $x_{\ell_1,\ell_2}=x_{\ell_1}^{i_{\ell_1}} x_{\ell_1+1}^{i_{\ell_1+1}} 
% \cdots x_{\ell_2}^{i_{\ell_2}}$. 
% \end{proposition}
% \begin{proof}
% {\color{blue}
% Observe first that an element in $A_\nu$ in eq. (\ref{A-categories}) can be written as
% \[
% x_{1,\nu-1}\cdot x_\nu^{i_\nu} x_{\nu+1,s-1} \cdot x_\nu\cdot 
% x_{\nu+1,s-1}^{-1}\cdot x_\nu^{-i_\nu-1}\cdot x_{1,\nu-1}^{-1} 
% =\]
% \[=
% \left(
% [x_\nu^{i_\nu},x_{\nu+1,s-1}]\cdot 
% [
% x_{\nu+1,s-1},x_\nu^{i_\nu+1}
% ] 
% \right)^{x_{1,\nu-1}}
% =[x_{\nu+1,s-1},x_\nu]^{x_{1,\nu-1}\cdot x_\nu^{i_\nu}},
% \]
% i.e. it can be expressed as the action of a commutator. 
% {\bf Already wrote this!}
% }
% Now observe that 
% \[
% [x\cdot y, z]=[x,z]^y \cdot [y,z].
% \]
% The above equation together with the identities given in lemma \ref{com-lem-ide} allows us to write:
% \todo{complete the argument}
% \end{proof}
\begin{remark}
For the homology of the closed curve we have:
\[
H_1( C_{k,s-1} ,\mathbb{Z})=\frac{R_{k,s-1}/R_{k,s-1}'}{\langle \gamma_1,\ldots,\gamma_{s k^{s-1}} \rangle}.
\]
Using eq. (\ref{genero}) and the fact that 
$\mathrm{rank} \;H_1( C_{k,s-1} ,\mathbb{Z})=2g_{ C_{k,s-1} }$
it is easy to verify that 
\begin{equation} \label{secondgenusFormula}
(s-2)k^{s-1}+1 -(s\cdot k^{s-2}-1)=2g_{ C_{k,s-1} }.
\end{equation}
In the above formula we have subtracted one from the number of invariant elements $\gamma_i$ since $\gamma_1\cdots \gamma_{sk^{s-2}}=1$.
\end{remark}
% \highlight[id=AK,comment=12/3/2019]{
Describing the action in this case is not as straightforward as it was for the case of classical Fermat curve. We will use the theory of  Alexander modules instead and postpone this computation to section \ref{sec:Crowell}.

\section{On the representation of Ihara}
\label{sec:IharaRep}
\subsection{Pro-$\ell$ braid groups}
Let $\ell$ be a prime number and let $\mathfrak{F}_{s}$
denote the pro-$\ell$ free group with $s$ free generators. Let $S\subset \mathbb{P}^1_{\bar{\Q}}$ be a set consisted of $s$ points, $s\geq 3$,  on the projective line  and suppose that  $P\in \mathbb{Q}$ for all $P\in S-\{\infty\}$. In this way the absolute Galois group corresponds to ``pure braids''. Ihara in \cite{Ihara1985-it} introduced the monodromy representation
\[
\mathrm{Ih}_S: \mathrm{Gal}(\bar{\Q}/\Q) \rightarrow \mathrm{Aut}(\mathfrak{F}_{s-1}).
\]
Here the  group 
$\mathfrak{F}_{s-1}
 \cong
\pi_1^{\mathrm{pro}-\ell}(\mathbb{P}^1_{\bar{\Q}}-S)$ is the pro-$\ell$ \'etale fundamental group and   is known to admit  a presentation
\begin{equation}
\label{Frpres}
\mathfrak{F}_{s-1}=
\reallywidehat{
\left\langle
x_1,\ldots,x_s | x_1x_2\cdots x_s=1
\right\rangle
}, 
\end{equation}
where $\reallywidehat{\cdot}$ denotes the pro-$\ell$ completion of a finitely generated group. 
Given a set $\{x_i, i\in I\}$ in a topological group we will denote by $\langle x_i, i \in I \rangle$
the topological closure of the group generated by the group elements $x_i$, $i\in I$.
In \cite{Ihara1985-it} Ihara studied the case $S=\{0,1,\infty\}$. This is an interesting case since by Belyi's theorem \cite{Belyi1} the branched covers of $\mathbb{P}^1 \backslash \{0,1,\infty\}$ are exactly the curves defined over $\bar{\Q}$. 
The case  $s\geq 3$ is also interesting and was also considered by Ihara, see \cite{MR1159208}. 
Using a M\"obious  transformation we can assume that the set $S$ consists of the elements $0,1,\lambda_1,\cdots,\lambda_{s-3},\infty$.

The Ihara representation can be explained in terms of Galois theory as follows:
 Consider the maximal pro-$\ell$ extension $\mathcal{M}$ of $\Q(t)$ unramified outside the set $S$.
The Galois group $\mathrm{Gal}(\mathcal{M}/\bar{\Q}(t))$ is known to be 
 isomorphic to
the pro-$\ell$ free group $\mathfrak{F}_{s-1}$ of rank $s-1$. A selection of generators $x_1,\ldots,x_{s-1}$ corresponds to an isomorphism
 $i:\mathfrak{F}_{s-1} \rightarrow \mathrm{Gal}(\mathcal{M}/\bar{\Q}(t))$, such that $i(x_\nu)$ ($1 \leq \nu \leq s$) generates the inertia group of some place $\xi_\nu$ of $\mathcal{M}$ extending the place $P_i$ of $\bar{\Q}(t)$, corresponding to the $i$-th element of the set $S$.

We have the following exact equence:
\begin{equation} \label{ActionFree}
  \xymatrix@R=12pt{
    1 \ar[r] & \mathrm{Gal}(\mathcal{M}/\bar{\Q}(t))
    \ar[r] \ar[d]_{i}^{\cong} &
    \mathrm{Gal}(\mathcal{M}/\Q(t)) \ar[r] &
    \mathrm{Gal}(\bar{\Q}(t)/\Q(t)) \ar[r] \ar[d]^{\cong} & 1 \\
    & \mathfrak{F}_{s-1} &  & \mathrm{Gal}(\bar{\Q}/\Q) &
  }
\end{equation}
Every element $\rho \in \mathrm{Gal}(\bar{\Q}/\Q)$ gives rise to an element $\rho^* \in \mathrm{Gal}(\mathcal{M}/\Q(t))$, 
 which is unique up to an element of $\mathrm{Gal}(\mathcal{M}/\bar{\Q}(t))$, and so we obtain
an 
isomorphism
 $x \mapsto \rho^* x \rho^{-1} \in \tilde{P}(\mathfrak{F}_{s-1})/\mathrm{Int}(\mathfrak{F}_{s-1})$, 
where 
\[
\tilde{P}(\mathfrak{F}_{s-1}):=
\left\{
\phi\in \mathrm{Aut}(\mathfrak{F}_{s-1})| \phi(x_i)\sim x_i^{N(\phi)}
(1\leq i \leq s) \text{ for some } N(\phi) \in \Z_\ell^*
\right\},
\]
and $\sim$ denotes the conjugation equivalence. 

Y. Ihara \cite[p.52]{Ihara1985-it}, proved that the action of $\sigma \in \mathrm{Gal}(\bar{\Q}/\Q)$ on the topological generators of $\mathfrak{F}_{s-1}$ is 
 in $\tilde{P}(\mathfrak{F}_{s-1})$ that is 
\[
\sigma(x_i)=w_i(\sigma) x_i^{N(\sigma)} w_i(\sigma)^{-1},
\]
where $N(\sigma)\in \Z_\ell^{*}$ and $w_i(\sigma) \in \mathfrak{F}_{s-1}$ is the element defining the conjugation.  
In this way the outer Galois representation 
\[
\Phi_S: \mathrm{Gal}(\bar{\Q}/\Q) \rightarrow 
\tilde{P}(\mathfrak{F}_{s-1})/\mathrm{Int}(\mathfrak{F}_{s-1})
\]
is defined.

 % for some  $a \in \mathbb{Z}_\ell$.
By selecting the representatives of elements $\tilde{P}(\mathfrak{F}_{s-1})$ we can define the Ihara representation 
\[
\mathrm{Ih}_S:\mathrm{Gal}(\bar{\Q}/\Q) \rightarrow P(\mathfrak{F}_{s-1})\subset \mathrm{Aut}(\mathfrak{F}_{s-1}),
\]
where 
\begin{equation}
P(\mathfrak{F}_{s-1})=
\left\{
\phi \in \mathrm{Aut}(\mathfrak{F}_{s-1})
\left|
\begin{array}{l}
\phi(x_i)\sim x_i^{N(\phi)}, (1\leq i \leq s-2), \phi(x_{s-1}) \approx x_{s-1}^{N(\phi)}
\\
\phi(x_s)=x_s^{N(\phi)}, \text{ for some }  N(\phi)\in \Z_\ell^{\times}
\end{array}
\right.
\right\},
\end{equation}
where $\approx$ denotes conjugacy by an element of the subgroup of 
$\mathfrak{F}_{s}$ generated by the commutator $\mathfrak{F}_s'$ and 
$x_1,\ldots,x_{s-3}$. 
%\highlight[id=AK,comment=write also the thorem on $N_1$]{
The composition $N\circ\mathrm{Ih}_S$ equals the cyclotomic character $\chi_\ell:\mathrm{Gal}(\bar{\Q}/\Q)\rightarrow \Z_\ell^*$.
For more details on these constructions see \cite[prop.3 p.55]{Ihara1985-it}, \cite[prop. 2.2.2]{MorishitaATIT}.
\subsection{Magnus embedding}
We will explain now 
the Magnus
embedding following \cite{MorishitaATIT}. This embedding is given by the map 
\begin{equation}
\label{MagnusEmbedd1}
\mathfrak{F}_{s-1} \rightarrow \mathbb{Z}_\ell[[u_1, u_2,
  \ldots, u_{s-1}]]_{\mathrm{nc}}
\end{equation}
of $\mathfrak{F}_{s-1}$
into the
``non-commutative'' formal power series algebra
$(x_i 
{\color{red}
\mapsto
}
 1+u_i$ for $1 \leq i \leq s-1)$. 
Let $\mathfrak{H}$ denote the abelianization of $\mathfrak{F}_{s-1}$, and $H$ the abelianization of $F_{s-1}$
\[
H:=\mathrm{gr}_1(F_{s-1})=H_1(F_{s-1},\mathbb{Z}) \qquad
\mathfrak{H}=:\mathrm{gr}_1(\mathfrak{F}_{s-1})=H_1(\mathfrak{F}_{s-1},\mathbb{Z}_\ell)=H\otimes_{\Z}\Z_{\ell}.
\]
The term $\mathrm{gr}_1$ above has its origin on the graded Lie algebra corresponding to a (pro-$\ell$) free group, see \cite[p. 58]{Ihara1985-it} and \cite{LazardGradedLie}.
Following \cite{MorishitaATIT}, \cite{Morishita2011-yw} we consider the tensor algebras 
\[
T(H)=\bigoplus_{n\geq 0} H^{\otimes n},
\qquad
T(\mathfrak{H})=\bigoplus_{n\geq 0} \mathfrak{H}^{\otimes n},
\]
where $\mathfrak{H}^0=\mathbb{Z}_\ell$ and $\mathfrak{H}^{\otimes n}:=\mathfrak{H} \otimes_{\mathbb{Z}_\ell} \cdots \otimes_{\mathbb{Z}_\ell} \mathfrak{H}$ ($n$-times)
(resp. $H^0=\Z$, $H^{\otimes n}=H\otimes_{\Z} \cdots \otimes_{\Z} H$) ).
If $u_0,\ldots,u_{s-1}$ is a $\mathbb{Z}_\ell$ basis of the free $\mathbb{Z}_\ell$-module $\mathfrak{H}$, then 
\[
T(\mathfrak{H})=\mathbb{Z}_\ell \langle u_1,\ldots,u_{s-1} \rangle,
\]
is the non-commutative polynomial algebra 
$\mathbb{Z}_\ell[[u_1, u_2,
  \ldots, u_{s-1}]]_{\mathrm{nc}}
  $
over $\mathbb{Z}_\ell$,  
 appearing in the  right hand side of eq. (\ref{MagnusEmbedd1}).

We will denote by $\widehat{T}(\mathfrak{H})$ the completion of $T(\mathfrak{H})$ with respect to the $\mathfrak{m}$-adic topology, where $\mathfrak{m}$ is the two sided ideal generated by $u_1,\ldots,u_{s-1}$ and $\ell$. This algebra is the algebra of non-commutative formal power series over $\mathbb{Z}_\ell$ with variables $u_1,\ldots,u_{s-1}$:
\[
\widehat{T}(\mathfrak{H})=\prod_{n\geq 0} 
\mathfrak{H}^{\otimes n}= \mathbb{Z}_\ell\langle\!\langle u_1,\ldots,u_{s-1} \rangle \! \rangle.
\]
Let $\mathbb{Z}_\ell[[\mathfrak{F}_{s-1}]]$ be the complete group algebra of $\mathfrak{F}_{s-1}$ over $\mathbb{Z}_\ell$, and let 
\[
\varepsilon_{\mathbb{Z}_\ell[[\mathfrak{F}_{s-1}]]}:\mathbb{Z}_\ell[[\mathfrak{F}_{s-1}]]  \rightarrow 
\mathbb{Z}_\ell
\]
be the augmentation homomorphism. Denote by $I_{\mathbb{Z}_\ell[[\mathfrak{F}_{s-1}]]}:=\ker \varepsilon_{\mathbb{Z}_\ell[[\mathfrak{F}_
{ s-1 }
]]}$ the augmentation ideal.
The correspodence $x_i \mapsto 1+u_i$ for $1 \leq i \leq s-1$ induces an isomorphism of topological $\mathbb{Z}_\ell$-algebras, the {\em pro-$\ell$ Magnus isomorphism}. 
\[
\Theta: \mathbb{Z}_\ell[[\mathfrak{F}_{s-1}]]\stackrel{\cong}{\longrightarrow} \widehat{T}(\mathfrak{H}).
\]
\begin{example}
The map $\Theta$ sends $\Z_\ell[\Z]=\Z_\ell[t,t^{-1}]$ to $\Z_\ell[[u]]$ by mapping $\Theta(t)=1+u$ and 
$\Theta(t^{-1})=(1+u)^{-1}=\sum_{i=0}^\infty (-1)^i u^i$.  The image $\Theta(\Z_\ell[t,t^{-1}])$ is not onto
 $
 % \Z_\ell[[1+u]]=
 \hat{T}(\mathfrak{H})$, but $\Z_\ell[[\Z_\ell]]=\Z_\ell[[\mathfrak{F}_1]]$ is mapped isomorphically to $\hat{T}(\mathfrak{H})$ by $\Theta$.
\end{example}

For an multiindex $I=(i_1,\ldots,i_{s-1})$ we set 
$u_I=u_{i_1}\cdots u_{i_{s-1}}$. The coefficient of $u_{I}$ in 
$\Theta(\alpha)$ is called the Magnus coefficient of $\alpha$ and it is denoted by $\mu(I,\alpha)$, that is
\[
\Theta(\alpha)=\varepsilon_{\Z_\ell[[\mathfrak{F}_{s-1}]]}(\alpha)
+
\sum_{|I| \geq 1} \mu(I,\alpha)u_I.
\]
For certain properties of the Magnus embedding and a fascinating application to $\ell$-adic Milnor invariants we refer to 
\cite[chap. 8]{Morishita2011-yw}, \cite[sec. 3.2]{MorishitaATIT}.
\subsection{Milnor invariants}
Consider the group $\mathfrak{H}:=\mathfrak{F}_{s-1}^{\mathrm{ab}}=\mathfrak{F}_{s-1}/[\mathfrak{F}_{s-1},\mathfrak{F}_{s-1}]$. For $f\in  \mathfrak{F}_{s-1}$ denote by $[f]$ its image in $\mathfrak{H}$. We will write $\mathfrak{H}$ as an additive $\mathbb{Z}_\ell$-module, which is  generated by $[u_1],\ldots,[u_{s-1}]$. Notice that the following relation holds:
\[
[u_1]+\cdots + [u_{s-1}]+[u_s]=0.
\]
Every automorphism $\phi\in \mathrm{Aut}(\mathfrak{F}_{s-1})$ gives rise to
a linear automorphism of the free $\mathbb{Z}_\ell$-module $\mathfrak{H}$ and
we will denote it by $[\phi]\in \mathrm{GL}(\mathfrak{H})$.
\begin{lemma}
The elements $w_i(\sigma)\in \mathfrak{F}_{s-1}$ can be selected uniquely so that 
\begin{enumerate}
\item $\mathrm{Ih}_S(\sigma)(x_i)=w_i(\sigma) x_i^{\chi_\ell(\sigma)} w_i(\sigma)^{-1}$, where $\chi_\ell$ is the $\ell$-cyclotomic character.
\item In the expression $[w_i(\sigma)]=c_1^{(i)} 
[u_1]+\cdots+ c_{s-1}^{(i)} [u_{s-1}], 
 c_j^{(i)}\in \Z_\ell$, we have $c_i^{(i)}=0$.
\end{enumerate}
\end{lemma}
\begin{proof}
See \cite[lemma 3.2.1]{MorishitaATIT}.
\end{proof}
For a multiindex $I=(i_1,\ldots,i_n)$, $1\leq i_1,\ldots,i_n \leq s-1$ the $\ell$-adic Milnor number for $\sigma\in \mathrm{Gal}(\bar{\Q}/\Q)$ is defined as the $\ell$-adic Magnus coefficient of $w_i(\sigma)$, for $I'=(i_1,\ldots,i_{n-1})$, that is
\[
\mu(\sigma,I):=\mu(I',w_{i_n}(\sigma)),
\]
see \cite[eq. 3.2.2]{MorishitaATIT}.
It is clear that the selection of $w_i(\sigma)$ describes completely the action of $\mathrm{Gal}(\bar{\Q}/\Q)$ on $\mathfrak{F}_{s-1}$.

\subsubsection{The commutative Magnus ring}
In this article we will consider actions of $\mathrm{Aut}(F_{s-1})$ or $\mathrm{Aut}(\mathfrak{F}_{s-1})$ on certain $\Z$-modules ($\Z_\ell$-modules) $M$ defined as quotients of subgroups of the (pro-$\ell$) free group. For example on  $F_{s-1}^{\mathrm{ab}}$ or on $\mathfrak{F}_{s-1}^{\mathrm{ab}}$. We would like for $M$ to be  an abelian group (we also choose to write $M$ additively) and we will entirely focus on the case $M=R/R'$, where $R< \mathfrak{F}_{s-1}$ (or $R < F_{s-1}$).

% \added[id=AK,comment=24/2/2019]{
The group $F_{s-1}$ (resp. $\mathfrak{F}_{s-1}$)
acts on itself by conjugation. This action can be translated as an $T(H)$ (resp. $\hat{T}(\mathfrak{H})$)
module structure  on $M$, by setting
\[
\alpha w \alpha^{-1}= \Theta(\alpha) \cdot w, 
\]
for $w\in F_{s-1}$ (resp. $w \in \mathfrak{F}_{s-1}$).
% }
% \deleted[id=AK]{
% By considering the action of $u_i\in T(H)$ on $F_{s-1}$ (resp, the action of $u_i\in \hat{T}(\mathfrak{H})$ on $\mathfrak{F}_{s-1}$) by conjugation by the element $x_i=1+u_i$ we can see that the conjugation action of the free group is equivalent to a $T(H)$-module (resp. $\hat{T}(\mathfrak{H})$-module) structure on $M$. 
% We would like to have the following property:}

\begin{lemma} \label{comm-action}
If $M=R/R'$ and 
$[R,\mathfrak{F}_{s-1}'] \subset R'$ (resp. $[R,F_{s-1}']\subset R'$)
then the induced conjugation action on $M$ satisfies  
\begin{equation}
\label{ab-action}
ab \cdot m=ba \cdot  m, \text{ for all } a,b \in \hat{T}(\mathfrak{H}) \text{(resp. $T(H))$  and } m \in M. 
\end{equation}
Notice that the inclusion $\mathfrak{F}_{s-1}' \subset R$ (resp. $F'_{s-1} \subset R$) implies the desired condition for the action to commute. 
\end{lemma}
\begin{proof}
For $a,b\in \mathfrak{F}_{s-1}$ and $r\in R$ we compute 
\[
ab r b^{-1} a^{-1}= ba [a^{-1},b^{-1}] r [a^{-1},b^{-1}]^{-1} a^{-1} b^{-1}.
\]
So a sufficient condition for eq. (\ref{ab-action}) to hold is  
$[R,\mathfrak{F}_{s-1}'] \subset R'$ (resp. $[R,F_{s-1}']\subset R'$). This condition is satisfied if $\mathfrak{F}_{s-1}' \subset R$ (resp. $F_{s-1} \subset R$) then eq. (\ref{ab-action}) holds.
\end{proof}
Therefore, if the assumption of lemma \ref{comm-action} holds, instead of considering the action of the non-commutative ring $\hat{T}(\mathfrak{H})$ (resp. $T(H)$) it makes sense to consider the action of the corresponding abelianized ring. 
\begin{definition} \label{Acomplete-def}
Consider the commutative $\mathbb{Z}_\ell$-algebra of formal power series
\begin{eqnarray}
\nonumber
\mathcal{A} & =& \mathbb{Z}_\ell[[u_i: 1\leq i \leq s]]/
\left\langle
(1+u_1)(1+u_2)\cdots(1+u_s)-1
\right\rangle \\
& \cong &
\mathbb{Z}_\ell[[u_i: 1\leq i \leq s-1]]. \label{A-def}
\end{eqnarray}
The algebra $\mathcal{A}$ is the symmetric algebra of $\mathfrak{H}$ over $\Z_\ell$, and there is a natural quotient  map 
$\hat{T}(\mathfrak{H}) \rightarrow \mathrm{Sym}(\mathfrak{H})=\mathcal{A}$.
\end{definition}
\begin{remark}
As we noticed already the action of $\sigma\in \mathrm{Gal}(\bar{\Q}/
\Q)$ can be described in terms of the cocycles $w_1(\sigma),\ldots,w_{s-1}(\sigma)$. But then we can find elements 
\[\varpi_1(\sigma)=\Theta(w_1(\sigma)),\ldots,\varpi_{s-1}(\sigma)=\Theta(w_{s-1}(\sigma))\in \mathcal{A}\] such that 
\begin{equation}
\label{nomodule}
\sigma(x_i) = \varpi_i(\sigma)\cdot x_i^{\chi_\ell(\sigma)}.
\end{equation}
Therefore, in order to understand the action of 
$\mathrm{Gal}(\bar{\Q}/\Q)$ on $M=\mathfrak{F}_{s-1}/\mathfrak{F}'_{s-1}$ 
% (keep in mind that $M$ is a quotient of  $\mathfrak{F}_{s-1}$) 
it makes sense to consider the $\mathcal{A}$-module structure of $M$.
\end{remark}

% \newpage
%
%
\section{Alexander modules}
\label{sec:Alexander}
\subsection{Definition and Crowell exact sequence}
\label{sec:Crowell}
We will use the notation of section \ref{sec:GeometInterpret} for the groups $\bar{R}_0, R=\bar{R}_0/\Gamma \cap \bar{R}_0 \cong \bar{R}_0\cdot \Gamma/\Gamma,\Gamma$. 
% Recall that we considered a Galois cover $\pi: Y_0$ is a topological cover of $X_S=\mathbb{P}^1_{\mathbb{C}} \backslash \{S\}$, where $S\subset \mathbb{P}^1_{\bar{\Q}}$ is a finite set of points of the projective line. By covering space theory, the fundamental group $R_0$ of $Y_0$ is a subgroup of $\pi_1(X_S)\cong \mathfrak{F}_{s-1}$. Recall that we also defined the group $\Gamma=\langle x_1^e_1,\ldots, x_s^{e_s} \rangle$, where $e_1,\ldots,e_s$ are the ramification indices of the ramification points of $\pi: \bar{Y} \rightarrow \mathbb{P}^1$. 
Consider the short exact sequence in eq. (\ref{short-def}).
The group $G=\mathfrak{F}_{s-1}/\Gamma$ admits the presentation:
\begin{equation}
\label{Ggroup-def}
G=
\reallywidehat{
\left\langle
x_1,\ldots,x_{s} |x_1^{e_1}=\cdots =x_{s}^{e_s}=x_1\cdots x_s=1 
\right\rangle.
}
\end{equation}
On the other hand since we  assumed that $\mathfrak{F}_{s-1}' \subset \bar{R}_0$, 
(see lemma \ref{comm-action})  
the group $\mathfrak{F}_{s-1}/\bar{R}_0\cdot \Gamma$ is isomorphic to a quotient of the  abelian group $\mathbb{Z}/e_1\mathbb{Z} \times\cdots \times \Z/e_{s-1}\Z$.

Recall that $\psi:\mathcal{F}_{s-1}/\Gamma \rightarrow \mathcal{F}_{s-1}/\bar{R}_0 \cdot \Gamma$. 
Set 
\[
% {\color{blue} \mathcal{A}:= }
\mathcal{A}^{\bar{R}_0,\Gamma}=\Z_\ell[[\mathcal{F}_{s-1}/\bar{R}_0\cdot \Gamma]],
\]
and define the map $\varepsilon_{\mathcal{A}^{\bar{R}_0,\Gamma} }:\Z_\ell[[\mathcal{F}_{s-1}/\bar{R}_0\cdot \Gamma]]\rightarrow \Z_\ell$ to be the augmentation map corresponding functorially  to the map $\mathcal{F}_{s-1}/\bar{R}_0\cdot \Gamma \rightarrow \{1_{\mathcal{F}_{s-1}/\bar{R}_0\cdot \Gamma}\}$, see \cite[8.3 p.99]{Morishita2011-yw}.

Consider also  $\mathcal{A}_\psi^{\bar{R}_0,\Gamma}$  to be the  {\em  Alexander module}, a free $\Z_\ell$-module
\[
\mathcal{A}_\psi^{\bar{R}_0,\Gamma}=
\left(
\bigoplus_{g\in \mathfrak{F}_{s-1}/\Gamma} 
\mathcal{A}^{\bar{R}_0,\Gamma} dg
\right)/
\big\langle 
d(g_1g_2)-dg_1-\psi(g_1) dg_2: g_1,g_2\in \mathfrak{F}_{s-1} /\Gamma
\big\rangle_{\mathcal{A}^{\bar{R}_0,\Gamma}}, 
\]
where the denominator in the above quotient denotes the $\mathcal{A}^{\bar{R}_0,\Gamma}$-module generated by the relations inside $\langle \ldots\rangle_{\mathcal{A}^{\bar{R}_0,\Gamma}}$.

Define also the  map $\theta_1: R^{\mathrm{ab}} \rightarrow \mathcal{A}_\psi^{\bar{R}_0,\Gamma}$ given by 
\begin{equation}
\label{theta1map}
R^{\mathrm{ab}} \ni n \mapsto dn
\end{equation}
and the map $\theta_2: \mathcal{A}^{\bar{R}_0,\Gamma}_{\psi}
\rightarrow
\mathcal{A}^{\bar{R}_0,\Gamma}$
to be  the homomorphism induced by 
\[
dg\mapsto \psi(g)-1 \text{ for } g\in G. 
\]

We will use the Crowell Exact sequence 
% {\color{red} explain the notation}
%\cite[chap. 9]{Morishita2011-yw} 
\cite[sec. 9.2, sec. 9.4]{Morishita2011-yw},
\begin{equation} \label{CrowellEx}
0 \rightarrow 
R^{\mathrm{ab}}
 =R/R'
\stackrel{\theta_1}{\longrightarrow} 
\mathcal{A}_\psi^{\bar{R}_0,\Gamma}
\stackrel{\theta_2}{\longrightarrow}
\mathcal{A}^{\bar{R}_0,\Gamma}
\stackrel{\varepsilon_{\mathcal{A}^{\bar{R}_0,\Gamma}  }}{\longrightarrow} 
\mathbb{Z}_\ell
\rightarrow 
0.
\end{equation}
% where 
% \[
% \mathcal{A}^{\bar{R}_0,\Gamma}=\Z_\ell[[\mathcal{F}_{s-1}/R_0\cdot \Gamma]],
% \]
% and  $\mathcal{A}_\psi^{\bar{R}_0,\Gamma}$  is the {\em  Alexander module}, a free $\Z_\ell$-module
% \[
% \mathcal{A}_\psi^{\bar{R}_0,\Gamma}=
% \left(
% \bigoplus_{g\in \mathfrak{F}_{s-1}/\Gamma} 
% \mathcal{A}^{\bar{R}_0,\Gamma} dg
% \right)/
% \big\langle 
% d(g_1g_2)-dg_1-\psi(g_1) dg_2: g_1,g_2\in \mathfrak{F}_{s-1} /\Gamma
% \big\rangle_{\mathcal{A}^{\bar{R}_0,\Gamma}}.
% \]
% {\color{red} The map $\varepsilon_{\mathcal{A}}$
% .......
% }
% The map $\theta_1: R^{\mathrm{ab}} \rightarrow \mathcal{A}_\psi^{\bar{R}_0,\Gamma}$
% is given by 
% \begin{equation}
% \label{theta1map}
% R^{\mathrm{ab}} \ni n \mapsto dn
% \end{equation}
% {\color{red}
% and the map $\theta_2$
% is the homomorphism induced by 
% \[
% dg\mapsto \psi(g)-1 \text{ for } g\in G. 
% \]
% }
% \highlight[id=AK,comment=13/3/2019]{Write more about the two representations the group algebra. Fourier transform?}
For a description of the Alexander module in terms of differentials in  non-commutative algebras we refer to \cite{ParamPartC19}. Notice that when the group  $\mathfrak{F}_{s-1}/\bar{R}_0\cdot \Gamma$ is finite then we will write $\Z_\ell[\mathfrak{F}_{s-1}/\bar{R}_0\cdot \Gamma]$ instead of  $\Z_\ell[[\mathfrak{F}_{s-1}/\bar{R}_0\cdot \Gamma]]$.
In this case $\varepsilon_{\mathcal{A}^{\bar{R}_0,\Gamma}  }$ is the augmentation map sending finite sums 
$\sum_{g\in \mathcal{F}_{s-1}/\bar{R}_0\cdot \Gamma} a_g g$ to $\sum_{g\in \mathcal{F}_{s-1}/\bar{R}_0\cdot \Gamma} a_g \in \Z_\ell$.

\begin{proposition}
\label{freeApsi}
The module $\mathcal{A}^{\bar{R}_0,\Gamma}_\psi$ admits the following free resolution as an 
$\mathcal{A}^{\bar{R}_0,\Gamma}$-module:
\begin{equation}
\label{Qfree-res}
\left( \mathcal{A}^{\bar{R}_0,\Gamma} \right)^{s+1}
\stackrel{Q}{\longrightarrow} 
\left( \mathcal{A}^{\bar{R}_0,\Gamma} \right)^{s}
\longrightarrow 
\mathcal{A}^{\bar{R}_0,\Gamma}_{\psi}
\longrightarrow 0,
\end{equation}
where $s$ is the number of generators of $G$, given in eq. 
(\ref{Ggroup-def}) and $s+1$ is the number of relations.
Let $\beta_1,\ldots,\beta_{s+1} \in \mathcal{A}^{\bar{R}_0,\Gamma}$.
 The map $Q$ is 
expressed in form of Fox derivatives \cite[sec. 3.1]{BirmanBraids},\cite[chap. 8]{Morishita2011-yw},  as follows 
\[
\begin{pmatrix}
\beta_1 \\
\vdots \\
\beta_{s+1}
\end{pmatrix}
\mapsto
\begin{pmatrix}
\psi \pi \left(\frac{\partial x_1^{e_1}}{\partial x_1}\right) & 
\psi \pi \left(\frac{\partial x_2^{e_2}}{\partial x_1}\right) &
\cdots
\psi \pi \left(\frac{\partial x_s^{e_s}}{\partial x_1}\right) &
\psi \pi \left(\frac{\partial x_1\cdots x_s}{\partial x_1}\right) \\
\psi \pi \left(\frac{\partial x_1^{e_1}}{\partial x_2}\right) & 
\psi \pi \left(\frac{\partial x_2^{e_2}}{\partial x_2}\right) &
\cdots
\psi \pi \left(\frac{\partial x_s^{e_s}}{\partial x_2}\right) &
\psi \pi \left(\frac{\partial x_1\cdots x_s}{\partial x_2}\right) \\
\vdots & \vdots &  & \vdots  \\
\psi \pi \left(\frac{\partial x_1^{e_1}}{\partial x_s}\right) & 
\psi \pi \left(\frac{\partial x_2^{e_2}}{\partial x_s}\right) &
\cdots
\psi \pi \left(\frac{\partial x_s^{e_s}}{\partial x_s}\right) &
\psi \pi \left(\frac{\partial x_1\cdots x_s}{\partial x_s}\right) \\
\end{pmatrix}
\begin{pmatrix}
\beta_1 \\
\vdots \\
\beta_{s+1}
\end{pmatrix},
\]
where $\pi$ is the natural epimorphism $\mathfrak{F}_s \rightarrow G$ defined by the presentation given in eq. (\ref{Ggroup-def}). 
\end{proposition}
\begin{proof}
See  \cite[cor. 9.6]{Morishita2011-yw}. 
\end{proof}
If in eq. (\ref{CrowellEx}) $\bar{R}_0=\mathfrak{F}'_{s-1}$ and $\Gamma=\{1\}$, then
$
\mathcal{A}^{\mathfrak{F}'_{s-1},\{1\}}=\mathbb{Z}_\ell [[u_1,\ldots,u_{s-1}]]=\mathcal{A},
$
as defined in eq. (\ref{A-def}). 

To summarize, for $H_0=\mathfrak{F}_{s-1}/\bar{R}_0\cdot \Gamma$, the Alexander module $\mathcal{A}^{R,\Gamma}_\psi$  can be computed as a cokernel of the function $Q$:
\begin{equation} \label{Alexander-module1}
\mathcal{A}_\psi^{\bar{R}_0,\Gamma}=
\mathrm{coker} \;{Q},
 \qquad
\left(
\mathcal{A}^{\bar{R}_0,\Gamma}
\right)^{s+1}=
 \Z_\ell[[H_0]]^{s+1} \stackrel{Q}{\longrightarrow}
 \Z_\ell[[H_0]]^{s}=
 \left(\mathcal{A}^{\bar{R}_0,\Gamma}\right)^{s}.
 % \stackrel{d_1}{\longrightarrow}
 % \Z_\ell[H_k].
\end{equation}
The exponents in the above formula reflect the fact that the group $G$ is generated by
$(s+1)$ relations over
 $s$ free 
variables. 

\begin{proposition}
\label{Blanchfield-Lyndon}
If $\Gamma=\{1\}$ in eq. (\ref{short-def}) the Crowell exact sequence gives the Blanchfield-Lyndon exact sequence:
\begin{equation}
\label{BlanchfieldL}
\xymatrix{
0 \ar[r] & R^{\mathrm{ab}}
\ar[r]
& 
\left(
\mathcal{A}^{\bar{R}_0,\{1\}}
\right)^{s-1}
\ar[r]^-{d_1} &
\mathcal{A}^{\bar{R}_0,\{1\}}
\ar[r]^-{\varepsilon
%_{\mathcal{A}^{\bar{R}_0,\{1\}}}
}
&
\mathbb{Z}_\ell
\ar[r] & 
0.
}
\end{equation}
\end{proposition}
\begin{proof}
See \cite[p.118]{Morishita2011-yw} for the discrete case and the pro-$\ell$ case follows similarly.
\end{proof}
\subsubsection{Alexander modules for generalised Fermat curves
}
% $\mathfrak{F}'_{s-1}/\mathfrak{F}''_{s-1}$ as an $\mathcal{A}$-module}
%
%
\label{FreedomOrDeath}
It is clear that the group $\mathfrak{F}'_{s-1}/\mathfrak{F}''_{s-1}$ is generated as an $\mathcal{A}$-module by the elements $[x_i,x_j]$ for $1\leq i < j \leq s-1$. 
% {\color{red} Cojugation action!}

In what follows $\mathfrak{F}_{s-1}'=\bar{R}_0$ in the context of eq. (\ref{short-def}).
The structure of $\mathfrak{F}'_{s-1}/\mathfrak{F}''_{s-1}$ as an $\mathcal{A}$-module is expressed in terms of the Crowell exact sequence, see section \ref{sec:Crowell}, related to the short exact sequence:
\[
1 \rightarrow \mathfrak{F}_{s-1}' \rightarrow 
\mathfrak{F}_{s-1} 
\stackrel{\psi}{\longrightarrow} \mathfrak{F}_{s-1}^{\mathrm{ab}} \rightarrow 1,
\]
\[
0 \rightarrow (\mathfrak{F}'_{s-1})^{\mathrm{ab}}=\mathfrak{F}_{s-1}'/\mathfrak{F}_{s-1}''
\rightarrow 
 \mathcal{A}_\psi
\rightarrow 
\mathbb{Z}_\ell[[u_1,\ldots,u_{s-1}]]
\rightarrow 
\mathbb{Z}_\ell
\rightarrow 
0,
\]
where $ \mathcal{A}_\psi= \mathcal{A}_\psi^{\mathfrak{F}_{s-1}',\{1\}}$ is the Alexander module and 
\[
\mathcal{A}=\mathcal{A}^{\mathfrak{F}_{s-1}',\{1\}}=\mathbb{Z}_\ell [[u_1,\ldots,u_{s-1}]].
\]
% The module $A_\psi$ can be described as the cokernel of the map:
% \[
% d_2:\mathcal{A}^s \rightarrow \mathcal{A}^r
% \]
% \[
% (b_1,\ldots,b_s) \mapsto \sum_{j=1}^r b_i (\psi\circ \pi)\frac{\partial R_i}{\partial x_j}
% \]
% \todo{correct this!}

%
%
\begin{example} \label{AlexFermat}
Assume that in eq. (\ref{short-def}) the group $H_0=(\Z/\ell^k\Z)^{s-1}$ and 
consider the open  generalized Fermat curve with fundamental group 
$\bar{R}_0=\mathfrak{F}_{s-1}'$.
Let $\mathfrak{R}_k=\Gamma$ be the smallest  closed normal subgroup of $\mathfrak{F}_{s-1}$ generated by $x_1^{\ell^k},\ldots,x_{s-1}^{\ell^k}$. 
The group $G=\mathfrak{F}_{s-1,k}=\mathfrak{F}_{s-1}/\mathfrak{R}_k$ admits the presentation: 
\[
\mathfrak{F}_{s-1,k}=
\left\langle
x_1,\ldots,x_{s} |x_1^{\ell^k}=\cdots =x_{s}^{\ell^k}=x_1\cdots x_s=1 
\right\rangle.
\]
Denote the images of the elements $x_i$ in $H_0$ by 
$\bar{x}_i$.
It is clear that $\mathcal{A}_\psi^{\mathfrak{F}'_{s-1},\mathfrak{R}_k}$ is a free $\Z_\ell$-module of rank
 \[
\mathrm{rank}_{\Z_\ell}(
\mathrm{coker}\; Q)
=
 s(\ell^k)^{(s-1)}-\mathrm{rank}_{\Z_\ell}(Q).
 \]
Observe that $\mathcal{A}^{\mathfrak{F}'_{s-1},\mathfrak{R}_k}\cong \Z_\ell[H_0]$
is a free $\Z_\ell$-module of rank 
$(\ell^k)^{s-1}$.
%  since it contains elements
% \begin{equation} \label{elinAlexmodule}
% \sum_{i_1,\ldots,i_{s-1}=0}^{\ell^k-1}
%  a_{i_1,\ldots,i_{s-1}}
%  \bar{x}_1^{i_1}\cdots \bar{x}_{s-1}^{i_{s-1}}, \text{ where } a_{i_1,\ldots,i_{s-1}} \in \Z_\ell.
% \end{equation}
% \todo{explain the passage from $x_i$ to $\bar{x}_i$}
By induction  we can prove
\begin{align}
\nonumber
\frac{\partial x_i^{\ell^k}}{\partial x_j} &= 
\delta_{ij}(1+x_i+x_i^2+\cdots + x_i^{\ell^k-1}) 
\text{ for } 1\leq j \leq s
\\
\frac{\partial x_1x_2\cdots x_{s}} {\partial x_j} &= x_1 \cdots x_{j-1}
\label{derive-prod}
\end{align}
Set $\Sigma_i=1+\bar{x}_i+\cdots+ \bar{x}_i^{\ell^k-1}$.
The map $Q$ in eq. (\ref{Alexander-module1})
is given by  the matrix 
on the left of the following equation
\cite[cor. 9.6]{Morishita2011-yw}
\begin{equation} \label{rankd2}
\begin{pmatrix}
\Sigma_1 &  0 & \cdots & 0 & 1 \\
0 & \Sigma_2  & \ddots & \vdots & \bar{x}_1 \\
\vdots & \ddots    & \ddots  & 0 &  \vdots \\ 
0 & \cdots  & 0 & \Sigma_s & \bar{x}_1 \bar{x}_2 \cdots \bar{x}_{s-1}
\end{pmatrix}
\begin{pmatrix}
\beta_1 \\ \vdots \\ \beta_{s+1}
\end{pmatrix}=
\begin{pmatrix}
\Sigma_1 \beta_1 + \beta_{s+1} \\
\Sigma_2 \beta_2 + \bar{x}_1\beta_{s+1} \\
\vdots\\
\Sigma_s \beta_s + \bar{x}_1\cdots \bar{x}_{s-1}\beta_{s+1}
\end{pmatrix}
\end{equation}
where $\beta_i\in \mathcal{A}^{\mathfrak{F}_{s-1}',\mathfrak{R}_k}$ for $1\leq i \leq s$.
% Let $\bar{x}_i$ denote the image of $x_i$  in the group $H$.
Observe that 
\[
\Sigma_i \bar{x}_i^\nu = \Sigma_i \text{ for all } 0\leq \nu \leq \ell^{k}-1.
\]
% {\color{blue}
% \highlight[id=AK,comment=13/3/2019]{Compute a resolution} The kernel of the map $Q$ consists of elements $\beta_1,\ldots,\beta_{s+1} \in k[H_0]$ satisfying the set of equations
% \begin{align*}
% \Sigma_1 \beta_1 &=  -\beta_{s+1} \\
% \Sigma_2 \beta_2 &=  -x_1 \beta_{s+1} \\
% \Sigma_3 \beta_3 &=  -x_1 x_2 \beta_{s+1}\\
% \vdots &\; \qquad \vdots \\
% \Sigma_s \beta_s &=  -x_1\cdots x_{s-1} \beta_{s+1} 
% \end{align*}
% Since $x_i \Sigma_i=x_i$ the above sets of equations are equivalent to 
% \[
% \Sigma_1 \beta_1=\Sigma_2 \beta_2 = \cdots \Sigma_s \beta_s =\beta_{s+1}.
% \]
% Now using that $(1-\bar{x}_i)\Sigma_i=1$, we obtain that 
% \[
% \mathrm{ker}Q= \mathrm{Im}(\mathrm{Q_0}) \]
% where 
% \begin{align*}
% Q_0: \Z_\ell[H_0] &\longrightarrow \Z_\ell[H_0]
% \\
% \beta &\longmapsto ( u_1 \beta, u_2 \beta, \ldots, u_s \beta, \beta  )
% \end{align*}
% Therefore the following exact sequence is a resolution for the Alexander module $\mathcal{A}_\psi^{\mathfrak{F}_{s-1}',\mathfrak{R}_k}$:
% \begin{equation}
% \label{resolAlex}
% 0 \rightarrow \Z_\ell[H_0]
% \stackrel{Q_0}{\longrightarrow} \Z_\ell[H_0]^{s+1}
% \stackrel{Q}{\longrightarrow} \Z_\ell[H_0]^s 
%  \rightarrow
% \mathcal{A}_\psi^{\mathfrak{F}_{s-1}',\mathfrak{R}_k}
% \rightarrow 0
% \end{equation}
% {\color{red}
% \[
% \mathrm{rank}_{\Z_\ell} \mathrm{Im}Q = 
% (s+1) (\ell^k)^{s-1} - (\ell^k)^s
% \]
% \[
% \mathrm{rank}_{\Z_\ell}Q=
% s(\ell^k)^{s-2}+(\ell^k)^{s-1}-1.
% \]
% }
% }

% \[
% (1+\bar{x}_i+\bar{x}_i^2+\cdots+\bar{x}_i^{\ell^k-1})\bar{x}_i^\nu=1+\bar{x}_i+\cdots+\bar{x}_i^{\ell^k-1}
% \]
\begin{lemma}
For $1\leq i \leq s-1$
the following equation holds
\[
\Sigma_i \cdot \Z_\ell[
H_0
% (\Z/\ell^k \Z)^{s-1}
]=
  \Sigma_i \cdot 
\Z_\ell \left[
\bigoplus_{
\substack{
\nu=1
\\
\nu\neq i 
}}^{s-1} 
\Z/\ell^k \Z
\right].
\]
On the other hand the module $\Sigma_s \Z_\ell[H_0]$ contains all elements invariant under the action of the product  $\bar{x}_1 \cdots \bar{x}_{s-1}$ and is a free $\Z_\ell$-submodule of $\Z_\ell[H_0]$. 
% \highlight[id=AK,comment=20/3/2019]{What about $\Sigma_s$?}
\end{lemma}
\begin{proof}
Write
\[
\Z_\ell[H_0]=
\Z_\ell \left[  
\bigoplus_{\nu=1}^{s-1} \Z/\ell^k \Z
\right]
=
\bigotimes_{\nu=1}^{s-1}
\Z_\ell \left[ 
 \Z/\ell^k \Z
\right].
\]
Therefore the multiplication by $\Sigma_i$ gives rise to the tensor product
\[
\left(
\bigotimes_{\nu=1}^{i-1}
\Z_\ell \left[ 
 \Z/\ell^k \Z
\right]
\right)
\bigotimes 
\left(
\Sigma_i \Z_\ell \left[ 
 \Z/\ell^k \Z
\right]
\right)
\bigotimes 
\left(
\bigotimes_{\nu=i+1}^{s-1}
\Z_\ell \left[ 
 \Z/\ell^k \Z
\right]
\right)=
\]
\[
\left(
\bigotimes_{\nu=1}^{i-1}
\Z_\ell \left[ 
 \Z/\ell^k \Z
\right]
\right)
\bigotimes 
\left(
\Sigma_i \Z_\ell
\right)
\bigotimes 
\left(
\bigotimes_{\nu=i+1}^{s-1}
\Z_\ell \left[ 
 \Z/\ell^k \Z
\right]
\right)
\]
and the desired result follows. 

For the case of $\Sigma_s \Z_\ell[H_0]$ invariance under the action of $\bar{x}_s=\bar{x}_1^{-1}\cdots\bar{x}_{s-1}^{-1}$ is clear. The rank computation follows by changing the basis of $H_0$ from $\bar{x}_1,\ldots,\bar{x}_{s-1}$ to the basis $\bar{x}_2,\ldots,\bar{x}_s$ and arguing as before. 
\end{proof}

% It follows an explicit proof of the above mentioned result which is less elegant and ommited.  
% Write $\beta_i$ as
% \begin{equation} \label{elinAlexmodule}
% \beta_i=
% \sum_{i_1,\ldots,i_{s-1}=0}^{\ell^k-1}
%  a_{i_1,\ldots,i_{s-1}}
%  \bar{x}_1^{i_1}\cdots \bar{x}_{s-1}^{i_{s-1}}, \text{ where } a_{i_1,\ldots,i_{s-1}} \in \Z_\ell.
% \end{equation}
% We compute 
% the product $\beta_i \Sigma_i$ 
% \begin{align*}
% \beta_i \Sigma_i &=
% \sum_{i_1,\ldots,
% \hat{i},\ldots,i_{s-1}=0}^{\ell^k-1}
% \bar{x}_1^{i_1} \cdots 
% \widehat{\bar{x}_i^{i_i}}
% \cdots
% \bar{x}_{s-1}^{i_{s-1}}
% \sum_{\nu=0}^{\ell^k-1} \Sigma_i \bar{x}_i^\nu
%  a_{i_1,\ldots,i_{s-1}}.
% \\
% % &=
% % \sum_{i_1,\ldots,
% % \hat{i},\ldots,i_{s-1}=0}^{\ell^k-1}
% % \bar{x}_1^{i_1} \cdots 
% % \widehat{\bar{x}_i^{i_i}}
% % \cdots
% % \bar{x}_{s-1}^{i_{s-1}}
% % \sum_{\nu=1}^{\ell^k-1} \Sigma_i
% %  a_{i_1,\ldots,i_{s-1}}
% % \\
% % &=
% % \sum_{i_1,\ldots,
% % \hat{i},\ldots,i_{s-1}=0}^{\ell^k-1}
% % \bar{x}_1^{i_1} \cdots 
% % \widehat{\bar{x}_i^{i_i}}
% % \cdots
% % \bar{x}_{s-1}^{i_{s-1}}
% %  a_{i_1,\ldots,i_{s-1}}.
% \end{align*}
% In the above product the $\widehat{\cdot}$ symbol denotes omitting the 
% corresponding factor.  

The image of the map $Q$ equals to the space generated by elements
\[
\begin{pmatrix}
\Sigma_1 \beta_1  \\
\Sigma_2 \beta_2  \\
\vdots\\
\Sigma_s \beta_s 
\end{pmatrix}
+
\begin{pmatrix}
1 \\
\bar{x}_1 \\
\vdots\\
\bar{x}_1\cdots \bar{x}_{s-1}
\end{pmatrix}
\beta_{s+1}.
\]
For different choices of 
$\beta_1,\ldots,\beta_s \in \mathcal{A}^{\mathfrak{F}'_{s-1},\mathfrak{R}_k}$
the first summand forms a free $\Z_\ell$-module of rank $s(\ell^k)^{s-2}$ and the second summand is a free $\Z_\ell$-module of rank 
$(\ell^k)^{s-1}$. 
Also their intersection is just $\Z_\ell$. 

Indeed, if for some $\beta_{1},\ldots,\beta_{s+1} \in \Z_\ell[H_0]$ we 
have 
\[\beta_{s+1} (1,\bar{x}_1,\ldots,\bar{x}_1\cdots \bar{x}_{s-1}) =   
( \Sigma_1 \beta_1, \ldots, \Sigma_s \beta_s)
\]
then  by comparison of the first coordinates we see that $\beta_{s+1}$ is invariant under the action of $\bar{x}_1$. So comparison of second coordinate gives us that 
$\bar{x}_1 \beta_{s+1} = \beta_{s+1}$ is invariant under the action of $\bar{x}_2$. By continuing this way we see that $\beta_{s+1}$ is invariant under the whole group $H_0$, that is $\beta_{s+1}$ belongs to the rank one $\Z_\ell$-module generated by $\Sigma_1 \Sigma_2 \cdots \Sigma_{s}$. 
In this way we see that 
\begin{lemma}
\label{IMQ}
\begin{equation}
\label{IMQeq}
\mathrm{Im}(Q) 
= 
\left(
\bigoplus_{\nu=1}^s \Sigma_i \Z_\ell[H_0] 
\right)
\bigoplus 
\Z_\ell[H_0]/\Z_\ell \Sigma_1\cdots \Sigma_{s}.
\end{equation}
Also 
\[
\mathrm{rank}_{\Z_\ell}Q=
s(\ell^k)^{s-2}+(\ell^k)^{s-1}-1.
\]
\end{lemma}
We would like to compute the cokernel of $Q$ 
as a
 $\Z_\ell[H_0]$-module. This computation lies within the theory of integral representation theory. 
This seems a very difficult problem since a complete set of representatives of the classes of indecomposable modules for groups of the form $(\Z/\ell^k \Z)^t$ seems to be known only for $t=1$ and  $k=1,2$,  see \cite{MR524365}. In this article we will not consider the problem in the integral representation setting and   instead we will consider the simpler problem of determination of the $H_0$-action on the space
$H_1( C_{k,s-1} ,\mathbb{F})$, where $\mathbb{F}$ is a field  which contains $\Z_\ell$ and  the $\ell^k$-roots of unity. 
 Let us fix a primitive $\ell^k$ root of unity $\zeta_{\ell^k}$. 
Set
\[
\mathcal{I}:= ( \Z \cap [0,\ell^k))^{s-1}.
\]
If $\mathbf{i}\in \mathcal{I}$, set  $i_s := i_1 + \cdots + i_{s-1}$. Now define 
\begin{equation}
\label{zDef}
z(\mathbf{i}):=\# \{
j:1 \leq j \leq s \text{ and } i_j \equiv 0 \mod \ell^k
\}.
\end{equation}
Now set
\begin{equation}
\label{smallccomp}
c_{\mathbf{i}}:=
\begin{cases}
1+z(\mathbf{i}) & z(\mathbf{i}) <s\\
z(\mathbf{i}) & z(\mathbf{i})=s
\end{cases}
\end{equation}
For an element $\mathbf{i}=(i_1,\ldots,i_{s-1})\in \mathbb{N}^{s-1}$ we define a character 
$\chi_{\mathbf{i}}$ on $H_0$ by
\[
\chi_{\mathbf{i}} 
\left(
\bar{x}_1^{\nu_1},\ldots,\bar{x}_{s-1}^{\nu_{s-1}}
\right)
=\zeta^{\sum_{\mu=1}^{s-1} \nu_\mu i_\mu}.
\]
We have the following 
\begin{lemma}
% Consider the character $\chi_{i_1,\ldots,i_{s-1}}$ on $H_0$ given by 
% \highlight[id=AK,comment=20/3/2019]{
  % The group is generated in $a_1,\ldots,a_{s-1}$!
% }
% \[
% \chi_{i_1,\ldots,i_{s-1}}(\bar{x}_1^{\nu_1},\ldots,\bar{x}_{s-1}^{\nu_{s-1}})=
% \zeta_{\ell^k}^{\sum_{\mu=1}^{s-1}  \nu_\mu i_\mu}.
% \]
% where $\zeta_{\ell^k}$ is a primitive $\ell^k$-root of unity. 
We have the following decomposition  
% \[
% \mathrm{Im}(Q)\otimes K=\bigoplus_{i_1,\ldots,i_{s-1}=0}^{\ell^k-1}
% c_{i_1,\ldots,i_{s-1}} \chi_{i_1,\ldots,i_{s-1}},
% \]
\[
\mathrm{Im}(Q) \otimes \mathbb{F} = \bigotimes_{\mathbf{i} \in \mathcal{I}} 
\mathbb{F}
c_{\mathbf{i}}
\chi_{\mathbf{i}},
\]
where $c_{\mathbf{i}}\in \mathbb{N}$ is the multiplicity of the corresponding character. 
% For the $s-1$-tuple of integers $\bar{i}=(i_1,\ldots,i_{s-1})$, $0\leq i_1,\ldots,i_{s-1} \leq \ell^k-1$ let $z\left(\bar{i}\right)$ be the number of $i_1,\ldots,i_{s-1},i_1+\cdots+i_{s-1}$ that are equal to zero modulo $\ell^k$.  We have
% \begin{equation}
% \label{smallccomp}
% c_{i_1,\ldots,i_{s-1}}=
% \begin{cases}
% 1+z\left(\bar{i}\right) & \text{ if }  0 \leq z\left(\bar{i}\right) \leq s-1 \\
% s=z\left(\bar{i}\right) & \text{ if } z\left(\bar{i}\right)=s.
% \end{cases}
% \end{equation}
\end{lemma}
\begin{proof}
Consider the decomposition given in lemma \ref{IMQ}.
The module $\mathbb{F}[H_0]$ contains once every possible character, therefore
% \[
% K[H_0]=\bigoplus_{i_1,\ldots,i_{s-1}=0}^{\ell^k-1}
%  K \chi_{i_1,\ldots,i_{s-1}}.
% \]
\[
\mathbb{F}[H_0] = \bigotimes_{\mathbf{i} \in \mathcal{I}} \mathbb{F} {\chi_{\mathbf{i}}}
\]
On the other hand the modules $\Sigma_i \mathbb{F}[H_0]$ for $0\leq i \leq s-1$ are trivially acted on by elements $\bar{x}_i$. This means that 
% \[
% \Sigma_i K[H_0]=\bigoplus_{\nu_1,\ldots,\hat{\nu}_i,\ldots,\nu_{s-1}=0}^{\ell^k-1}
%  K \chi_{\nu_1,\ldots,\nu_{i-1},0,\nu_{i+1},\ldots,\nu_{s-1}}.
% \]
\[
\Sigma_i \mathbb{F}[H_0]=\bigoplus_{
\substack{
\mathbf{i} \in \mathcal{I}
\\
\mathbf{i}=(\nu_1,\ldots,\nu_{i-1},0,\nu_{i+1},\ldots,\nu_{s-1})
}
}
 \mathbb{F} \chi_{\mathbf{i}}.
\]
Also the module $\Sigma_s \mathbb{F}[H_0]$ contains elements which are invariant by elements of the group generated by
$\bar{x}_1\cdots \bar{x}_{s-1}$, since $\bar{x}_s=\bar{x}_1^{-1} \cdots \bar{x}_{s-1}^{-1}$. This means that all characters which appear in the decomposition of $\Sigma_s \mathbb{F}[H_0]$ on $\bar{x}_1^\nu \cdots \bar{x}_{s-1}^\nu$ should 
give $1$, which is equivalent to 
% \[
% \chi_{i_1,\ldots,i_{s-1}}(\bar{x}_1^\nu \cdots \bar{x}_{s-1}^\nu)=\zeta^{\nu \sum_{\mu=1}^{s-1} i_\nu}=1 \Rightarrow \sum_{\mu=1}^{s-1}i_\nu\equiv 0 \mod \ell^k.
% \]
\[
\chi_{\mathbf{i}}(\bar{x}_1^\nu \cdots \bar{x}_{s-1}^\nu)=\zeta_{\ell^k}^{\nu \sum_{\mu=1}^{s-1} i_\mu}=1 \Rightarrow \sum_{\mu=1}^{s-1}i_\mu\equiv 0 \mod \ell^k.
\]
Therefore, the decomposition into characters is given by
\[
\Sigma_s \mathbb{F}[H_0]=\bigoplus_{
  \substack{
  \mathbf{i} \in \mathcal{I} \\
  i_1,\ldots,i_{s-1}=0 \\
  i_1+\cdots+ i_{s-1}=0
  }
}
% ^{\ell^k-1}
\mathbb{F} \chi_{\mathbf{i}}.
\]
Given a character $\chi_{\mathbf{i}}$ we now count the number of times it appears. It appears on the  summands $\Sigma_j \mathbb{F}[H_0]$ for $0 \leq j \leq s-1$ when $i_j=0$ and in the summand $\Sigma_s \mathbb{F}[H_0]$ when $i_1+\cdots+i_{s-1}\equiv 0 \mod \ell^k$. Also it appears on $\mathbb{F}[H_0]/\Sigma_1\cdots \Sigma_{s}$ only if $(i_1,\ldots,i_{s-1})\neq (0,\ldots,0)$.
\end{proof}
\begin{lemma}
\label{alexHomodeco}
We have
\begin{align}
\label{rankAlexFermat}
\mathrm{rank}_{\Z_\ell} 
\mathcal{A}_\psi^{\mathfrak{F}_{s-1}',\mathfrak{R}_k} &=(s-1)(\ell^k)^{s-1}-s(\ell^k)^{s-2}+1
\end{align}
and 
\begin{equation}
\label{eigenAdeco}
\mathcal{A}_\psi^{\mathfrak{F}_{s-1}',\mathfrak{R}_k} \otimes \mathbb{F}=
\bigoplus_{\mathbf{i} \in \mathcal{I}}
% ^{\ell^k -1}
(s-c_{\mathbf{i}})
\chi_{\mathbf{i}}.
\end{equation}
\end{lemma}
\begin{proof}
The rank computation follows since $\mathcal{A}_\psi^{\mathfrak{F}_{s-1}',\mathfrak{R}_k}$ is the 
cokernel of $Q$, so 
\[
\mathrm{rank}_{\Z_\ell} 
\mathcal{A}_\psi^{\mathfrak{F}_{s-1}',\mathfrak{R}_k}=
s(\ell^k)^{s-1}-s(\ell^k)^{s-2}-(\ell^k)^{s-1}+1
=
 (s-1)(\ell^k)^{s-1}-s(\ell^k)^{s-2}+1.
\]
 Similarly the decomposition in eq. (\ref{eigenAdeco}) follows by the decomposition of $\mathbb{F}[H_0]$ into characters. 
\end{proof}
% \[
% \mathrm{rank}(d_2)=(\ell^{k})^{s-1}+ s(\ell^k)^{s-2}-1.
% \]
% \[
% (s-2)(\ell^k)^{s-1}+1 -(s\cdot (\ell^k)^{s-2}-1)=2g_{ C_{k,s-1} }.
% \]

\begin{proposition}
\label{prop35}
\[
H_1( C_{\ell^k,s-1} ,\mathbb{F})=\bigoplus_{\mathbf{i} \in \mathcal{I}}
\mathbb{F}
%^{\ell^k-1}
C(\mathbf{i}) \chi_{\mathbf{i}}
\]
where 
\begin{equation}
\label{Czeq}
C(\mathbf{i})=
\begin{cases}
s-c_{\mathbf{i}}-1 =s -z(\mathbf{i})-2 & \text{if } 
\mathbf{i}
\neq (0,\ldots,0)
\\
s-c_{\mathbf{i}}=s-z(\mathbf{i}) & \text{if } 
\mathbf{i}=
(0,\ldots,0)
\end{cases}
\end{equation}
% \[
% C(\mathbf{i})=
% \begin{cases}
% s-z(i_1,\ldots,i_s)-2 & \text{if } \mathbf{i}
% \neq (0,\ldots,0)
% \\
% s-z(i_1,\ldots,i_s) & \text{if } \mathbf{i}=
% (0,\ldots,0)
% \end{cases}
% \]
Moreover 
\[
\mathrm{rank}_{\Z_\ell} H_1( C_{\ell^k,s-1} ,\Z_\ell)=
(s-1)\left( \ell^k \right)^{s-1} +2 -s \left(\ell^k \right)^{s-2}.
\]
\end{proposition}
\begin{proof}
From the exact sequence given in eq. (\ref{CrowellEx}) and the rank computation given in eq. (\ref{rankAlexFermat}) in example \ref{AlexFermat} we have:
% \begin{align*}
% 2g_{F,{k,s-1}}&= s(\ell^{k})^{s-1}- (\ell^k)^{s-1}+1 -\mathrm{rank}(d_2) \\
% &= (s-1)(\ell^{k})^{s-1}+1 -s (\ell^k)^{s-2} -(\ell^k)^{s-1}+1 \\
% &= (s-2)(\ell^{k})^{s-1}+2 -s (\ell^k)^{s-2}
% \end{align*}
\begin{align}
\label{CrowellGFMgenus}
% 2g_{ C_{k,s-1} }
\mathrm{rank} 
\big(
R_{\ell^k}/(\mathfrak{R}_k \cap R_{\ell^k}) 
\big)^{\mathrm{ab}}
&= \mathrm{rank}_{\Z_\ell} \mathcal{A}^{\mathfrak{F}'_{s-1},\mathfrak{R}_k}_\psi -
\mathrm{rank}_{\Z_\ell} \mathcal{A}^{\mathfrak{F}'_{s-1},\mathfrak{R}_k}+1\\
% &= (s-1)(\ell^{k})^{s-1}+1 -s (\ell^k)^{s-2} -(\ell^k)^{s-1}+1 \\
&= (s-2)(\ell^{k})^{s-1}+2 -s (\ell^k)^{s-2}. \nonumber
\end{align}
The above abelianization corresponds to 
the $\Z_\ell$-homology of the generalized Fermat curves of type $(k,s-1)$. The above rank coincides with 
 the
genus computation  given in eq. (\ref{secondgenusFormula}).

Let us write 
\[
H_1( C_{\ell^k,s-1} ,\mathbb{F})=
\bigoplus_{\mathbf{i} \in \mathcal{I}}
% ^{\ell^k-1}
\mathbb{F}
C(\mathbf{i}) \chi_{\mathbf{i}}
\]
for some integers $C(\mathbf{i})$. 
By lemma \ref{alexHomodeco} and the short exact sequence given in (\ref{CrowellEx}) we obtain eq. (\ref{Czeq}).
% The first assertion follows by the values of $c(\mathbf{i})$ given in eq. (\ref{smallccomp}).
\end{proof}
\end{example}
\begin{remark}
For the case of classical Fermat curves we have $s=3$. 
% {\color{red} 
% Let $z(\nu,\mu)$ denote denote the number of entries in $\mathbf{i}=(\nu,\mu,\nu+\mu)\in \mathcal{I}$ that are equal to zero modulo $\ell^k$. 
% }
The character $\chi_{0,0,0}$ has $z(0,0,0)=3$ and $C_{0,0,0}=0$. Similarly the characters $\chi_{0,i,i},\chi_{i,0,i}$ for $1\leq i \leq \ell^k-1$ and the character $\chi_{i,j,i+j}$ with $i+j\equiv 0 \mod \ell^k$ have $z(0,i,i)=z(i,0,i)=z(i,j,i+j)=1$ so their contribution is $C(0,i,i)=C(i,0,i)=C(i,j,i+j)=0$. All other characters $\chi_{(i,j,i+j)}$ have $z(i,j,i+j)=0$ and their contribution is $C(i,j,i+j)=1$. In this way we arrive to the same result as in eq. (\ref{FermatDecoK}). 
\end{remark}

\begin{example} \label{examp18}
Let us now compute $\mathcal{A}_\psi^{R_{\ell^k},\mathfrak{R}_k}$
%  where
%  $\bar{R}_{0}$ is  the pro-$\ell$ completion of the group given by
% \[
% R_0=
% \{x_1^i x_j x_1^{-i-1}: 2 \leq j, i\in \mathbb{Z}\}
% \]
and $R_{\ell^k}
% =R_0/\mathfrak{R}_{\ell^k}
$ is the the pro-$\ell$ completion of the group generated  by
\[
 \{ x_1^i x_j x_1^{-i-1}: 2 \leq j \leq s-1,0 \leq i \leq \ell^k-2\}
 \cup 
 \{ x_1^{\ell^k-1}x_j: 1 \leq j \leq s-1 \}.
\]
This group corresponds to the open cyclic cover of order $\ell^k$ of $\mathbb{P}^1$ ramified fully above $s$-points of the projective line, see \cite[lemma 11]{MR4117575}.
Let $\mathfrak{R}_k=\Gamma$ be the smallest  closed normal subgroup of $\mathfrak{F}_{s-1}$ generated by $x_1^{\ell^k},\ldots,x_{s-1}^{\ell^k}$.
% and $R_{\ell^k}=R_{0,\ell^k}/\mathfrak{R}_k$.
We have the short exact sequence 
\[
1 \rightarrow R_{\ell^k}/R_{\ell^k} \cap \mathfrak{R}_k
% =R_{0}/\mathfrak{R}_k
 \rightarrow \mathfrak{F}_{s-1}/\mathfrak{R}_k 
\rightarrow
\Z/ \ell^k \Z 
\rightarrow 0.
\]
% \todo{what is $R_{0,\ell^k}$?}
We compute  $\mathcal{A}^{R_{\ell^k},\mathfrak{R}_k}=\Z_\ell[\Z/\ell^k\Z]$, which is an $\Z_\ell$-module of rank $\ell^k$. On the other hand observe that the $\Z_\ell$-module $\mathcal{A}_\psi^{R_{\ell^k},\mathfrak{R}_k}$ is given by exactly the same cokernel as the module $\mathcal{A}^{\mathfrak{F}_{s-1}',\mathfrak{R}_k}$. The only difference is that  $\mathcal{A}^{\mathfrak{F}'_{s-1},\mathfrak{R}_k}_\psi$ is a
 $\Z_\ell[(\Z/\ell^k\Z)^{s-1}]$-module while
  $\mathcal{A}^{R_{\ell^k},\mathfrak{R}_k}$ is a 
$\Z_\ell[\Z/\ell^k\Z]$-module. 

So following exactly the same method as in example \ref{AlexFermat} we conclude that
\[
\mathrm{rank}_{\Z_\ell} 
\mathcal{A}_\psi^{R_{\ell^k},\mathfrak{R}_k}=
s\cdot \ell^k-s-\ell^k+1
=
 (s-1)\ell^k-s+1.
\]
Also we compute the rank
% \begin{align*}
% 2g_{F,{k,s-1}}&= s(\ell^{k})^{s-1}- (\ell^k)^{s-1}+1 -\mathrm{rank}(d_2) \\
% &= (s-1)(\ell^{k})^{s-1}+1 -s (\ell^k)^{s-2} -(\ell^k)^{s-1}+1 \\
% &= (s-2)(\ell^{k})^{s-1}+2 -s (\ell^k)^{s-2}
% \end{align*}
\[
\mathrm{rank}
\big(
R_{\ell^k}/(R_{\ell^k} \cap \mathfrak{R}_k)
\big)^{\mathrm{ab}}
= \mathrm{rank}_{\Z_\ell} \mathcal{A}^{R_{\ell^k},\mathfrak{R}_k}_\psi -
\mathrm{rank}_{\Z_\ell} \mathcal{A}^{R_{\ell^k},\mathfrak{R}_k}+1
% &= (s-1)(\ell^{k})^{s-1}+1 -s (\ell^k)^{s-2} -(\ell^k)^{s-1}+1 \\
= (s-2)(\ell^{k}-1). 
\] 
The module $\big(
R_{\ell^k}/R_{\ell^k} \cap \mathfrak{R}_k
\big)^{\mathrm{ab}}$
corresponds to the $\Z_\ell$-homology of the above 
curves, corresponding to $R_{\ell^k}$ and its rank is twice the genus of the curve, in accordance with the genus formula given in \cite[eq. 21]{MR4117575}.
\end{example}

\section{Galois  modules in terms of the Magnus embedding}
\label{sec:Magnus}

\subsection{The group $\mathfrak{F}_{s-1}'/\mathfrak{F}_{s-1}''$ as an $\mathcal{A}$-module}
\label{sec.5.1}
In this section we will study 
the  $\mathcal{A}$-module structure of $\mathfrak{F}_{s-1}'/\mathfrak{F}_{s-1}''$. This is the arithmetic analogon of the Gassner representation, as Ihara points out in \cite{IharaCruz}. This consideration leads to the Galois representation of the Tate module, see section \ref{JacobGenFermat}. Finally in section 
\ref{Gassner2Burau} we will study the passage from the Gassner representation to the Burau by seeing the generalized Fermat curve as a cover of the projective line.  
% \begin{itemize}
% 	\item {\color{red} Relate group theory to the Jacobian Variety of the
% 	Generalized Fermat curve}
% 	\item {\color{red} Seek for the Braid group analogon }
% \end{itemize}

%
%
\subsubsection{Application to Generalized Fermat curves}
Consider the  the smallest closed normal subgroup $\mathfrak{R}_k$
of $ \mathfrak{F}_{s-1}$
containing  all $x_i^{\ell^k}$  for $1 \leq i \leq s-1$.
Define also 
\[
\mathfrak{F}_{s-1,k}=\mathfrak{F}_{s-1}/\mathfrak{R}_k.
\]
Set $\bar{\lambda}=\{0,1,\infty,\lambda_1,\ldots,\lambda_{s-3}\}$ and 
let $\mathcal{M}$ be the maximum pro-$\ell$ extension of
$K=\bar{\mathbbm{k}}(t)$ 
unramified
outside the set of  points $\bar{\lambda}$.
Consider the function field of the generalized Fermat curves
\[
	K_k:=K\left(
	t^\frac{1}{\ell^k}, (t-1)^{1/\ell^k}, (t-\lambda_1)^{1/\ell^k},\ldots,
	(t-\lambda_{s-3})^{1/\ell^k}
	\right).
\]
Let $K_k^{\mathrm{ur}}$ and $K_k^{\mathrm{unrab}}$ be the maximal unramified and maximal abelian unramified extensions of $K_k$ respectively.
Also let $K'$ be the maximum abelian unramified extension of $K$ and $K''$ be the maximum abelian unramified extension of $K'$. By covering space theory, the fields $K'$, $K''$ correspond to the groups $\mathfrak{F}_{s-1}'$ and $\mathfrak{F}_{s-1}''$, respectively. The function field $K_k$ corresponds to the group $\mathfrak{F}_{s-1}'\mathfrak{R}_k$ and is equal to the function field of the generalized Fermat curve. 

 The aim
 of this section is the following characterization of the maximal unramified abelian extension $K_k^{\mathrm{unrab}}$ of the function field $K_k$ of the generalized Fermat curve. This is a generalisation of a similar construction by Ihara for the classical Fermat curves, see \cite[sec. II, p. 63]{Ihara1985-it}
\begin{theorem}
We have that $\mathrm{Gal}(K_k^{\mathrm{unrab}}/K_k)\cong \mathfrak{F}_{s-1,k}'/\mathfrak{F}_{s-1,k}''$.
\end{theorem}

Indeed,
we have
\begin{align*}
K' &=  \bigcup_{k} K_k,  & K' \cap K_k^{\mathrm{ur}} & =K_k
\\
K'' & =  \bigcup_k K_k^{\mathrm{unrab}},  & K'' \cap K_k^{\mathrm{ur}} &=K_k^{\mathrm{unrab}}
\end{align*}
The Galois correspondence is given as follows:
\[
	\xymatrix{
		&      \mathcal{M} \ar@{-}[ld]  \ar@{-}[rd] & \\
		K'' \ar@{-}[d] \ar@{-}[drr]&  & K_k^{\mathrm{ur}} \ar@{-}[d] \\
		K' \ar@{-}[rd] &  &   K_k^{\mathrm{unrab}} \ar@{-}[ld]^{\mathrm{Gal}(K_k^{\mathrm{unrab}}/K_k)}  \\
		& K_k\ar@{-}[d]  & \\
		&  K      &
	}
	\qquad
	\xymatrix{
		&      \{1\} \ar@{-}[ld]  \ar@{-}[rd] & \\
		\mathfrak{F}_{s-1}'' \ar@{-}[d] \ar@{-}[drr]&  & \mathfrak{R}_k \ar@{-}[d] \\
		\mathfrak{F}_{s-1}' \ar@{-}[rd] &  &   \mathfrak{F}_{s-1}''\mathfrak{R}_k  \ar@{-}[ld] \\
		&  \mathfrak{F}_{s-1}'\mathfrak{R}_k \ar@{-}[d]   & \\
		&  \mathfrak{F}_{s-1}     &
	}
\]
Using standard isomorphism theorems in group theory (see also \cite[sec. 1.2.1]{ParamPartC19}) and the definitions we see
\begin{equation} \label{unrabGFC}
    \mathfrak{F}_{s-1,k}'/\mathfrak{F}_{s-1,k}''\cong
	\mathfrak{F}_{s-1}'/ \left( \mathfrak{F}_{s-1}' \cap
	\mathfrak{F}_{s-1}'' \mathfrak{R}_k  \right)
	\cong
	\mathfrak{F}_{s-1}' \mathfrak{R}_k /\mathfrak{F}_{s-1}'' \mathfrak{R}_k
	\cong
	\mathrm{Gal}(K_k^{\mathrm{unrab}}/K_k)
\end{equation}
is an abelian group, a free $\mathbb{Z}_\ell$-module of rank
$2g_{(\ell^k,s-1)}$, where $g_{(\ell^k,s-1)}$ is the genus of the generalized
Fermat curve, $F_{\ell^k,s-1}$ so that 
\begin{equation}
	% 2g_{(k,n)}= \text{\hcancel{$\ell^{k(n-1)}(n-1)(\ell^{k}-1)$}}
	% {\color{red} \ell^{k(n-1)}
	% \big(
	% (n-1)(\ell^{k}-1)-2
	% \big)+2}.
	2g_{(\ell^k,s-1)}=2+ \ell^{k(s-2)} ((s-2)(\ell^k-1)-2).
\end{equation}
Observe that according to eq. (\ref{genFermatCurveFuchs}) we have
\[
\mathfrak{F}_{s-1,k}'/\mathfrak{F}_{s-1,k}''\cong H_1( C_{\ell^k,s-1} ,\Z_\ell).
\]
The last genus computation also  follows from the following proposition which identifies unramified $\mathbb{Z}/\ell^k\mathbb{Z}$-extensions of a curve $X$ with the group of $\ell^k$-torsion points of the Jacobian $J(X)$. 

\begin{proposition}
Let $Y$ be a complete nonsingular algebraic curve defined over a field of characteristic prime to $\ell$. The \'etale Galois covers of $Y$ with Galois group $\mathbb{Z}/\ell^k \mathbb{Z}$ are classified by the \'etale cohomology group $H^1_{\mathrm{et}}(Y,\mathbb{Z}/\ell^k\mathbb{Z})$ which is equal to the group of $\ell^k$-torsion points of $\mathrm{Pic}(Y)$.
\end{proposition}
\begin{proof}
See \cite[Ex. 2.7]{Hartshorne:77}, \cite[sec. 19]{SerreMexico}.
\end{proof}

\subsubsection{Crowell sequence for generalized Fermat curves}
\label{alexGFC}
 We will use the notation of section \ref{sec:Alexander}
where $\bar{R}_0=\mathfrak{F}_{s-1}'$ and $\Gamma=\mathfrak{R}_k$, $R=\mathfrak{F}_{s-1}'/(\mathfrak{R}_k \cap \mathfrak{F}_{s-1})\cong \mathfrak{F}_{s-1}' \cdot \mathfrak{R}_k/\mathfrak{R}_k$. 
Notice also that the commutator identity for quotients imply that  $\mathfrak{F}_{s-1,k}'=\mathfrak{F}_{s-1}' \cdot \mathfrak{R}_k/\mathfrak{R}_k=R$. 
Here we use the presentation $\mathfrak{F}_{s-1}\cong\mathfrak{F}_{s}/\langle x_1\cdots x_s \rangle$.
Let $H_k=(\Z/\ell^k \Z)^{s-1}$. 
We have the short exact sequence
\[
1 \rightarrow \mathfrak{F}_{s-1,k}'=(\mathfrak{F}_{s-1}/\mathfrak{R}_k)' \rightarrow 
\mathfrak{F}_{s-1,k}=\mathfrak{F}_{s-1}/\mathfrak{R}_k \stackrel{\psi}{\longrightarrow} 
H_k \rightarrow 1
\]
The Crowell Exact sequence see eq. (\ref{CrowellEx}) and \cite[chap. 9]{Morishita2011-yw} gives us
\begin{equation} \label{CrowellEx1}
\xymatrix@C-10pt{
0 \ar[r] & 
(\mathfrak{F}'_{s-1,k})^{\mathrm{ab}}=\mathfrak{F}_{s-1,k}'/\mathfrak{F}_{s-1,k}''
\ar[r]^-{\theta_1} 
&
\mathcal{A}_\psi^{\mathfrak{F}_{s-1}',\mathfrak{R}_k}
\ar[r]^{\theta_2}
&
\mathcal{A}^{\mathfrak{F}_{s-1}',\mathfrak{R}_k}
\ar[r]^-
{\varepsilon_
{
  \mathcal{A}_k
}}
&
\mathbb{Z}_\ell
\ar[r] &
0,
}
\end{equation}
where 
\[
\mathcal{A}^{\mathfrak{F}_{s-1}',\mathfrak{R}_k}=\Z_\ell[H_k]=\Z_\ell[(\Z/\ell^k \Z)^{s-1}],
\]
\begin{equation} \label{Alexander-module}
\mathcal{A}_\psi^{\mathfrak{F}_{s-1}',\mathfrak{R}_k}=\mathrm{coker}\;{Q}, \qquad
 \Z_\ell[H_k]^{s+1} \stackrel{Q}{\longrightarrow}
 \Z_\ell[H_k]^{s}  
 % \stackrel{d_1}{\longrightarrow}
 % \Z_\ell[H_k].
\end{equation}
and $\varepsilon_{\mathcal{A}_k}$ is the augmentation map. 
The Alexander module for $\mathfrak{F}_{s-1}/\mathfrak{R}_k$ was computed on example \ref{AlexFermat}. Notice that $\mathcal{A}_\psi^{\mathfrak{F}_{s-1},\mathfrak{R}_k}$ and the Crowell sequence know the genus of the generalised Fermat curve, see eq. (\ref{CrowellGFMgenus}).

\subsubsection{Representation theory on  Generalized Fermat Curves}

Let $G$ be one of the groups $G_\mathbb{Q},B_{s-1}$ or $B_{s}$. 
These are   representations  on the free $\Z_\ell$-modules 
\[
\rho_k: G \rightarrow \mathrm{GL} (H_1( C_{\ell^k,s-1} ,\Z_\ell)).
\]
% Where $G$ is either the absolute Galois group or the braid group $B_{s-1}$ or $B_{s}$.

% \begin{enumerate}
% 	\item 
% {\color{red} Understand the limit procedure. The group $G$ should be replaced by the Galois group expressing the kernel. We should somehow see $B_s$ as a limit $q \rightarrow 1 $}
% \item 
% {\color{red} Also understand the reduction from the Gassner to the Generalized Fermat. This happens by introducing certain kernels, which are computed in Ihara's article. }
% \item 
% {\color{red} 
% Put in the same setting Braid and Galois representations. $B_2,B_3$ visible in the generalized Fermat setting are not interesting from the Braid point of view. 
% }
% \end{enumerate}
Let us now combine the two Crowell sequences together.  
\begin{equation} \label{reductionA}
\xymatrix@C=10pt{
0 \ar[r] &  (\mathfrak{F}'_{s-1})^{\mathrm{ab}}=\mathfrak{F}_{s-1}'/\mathfrak{F}_{s-1}''
\ar[r] \ar[d]^{\phi}& 
\mathcal{A}^{s-1}
\ar[r]^{d_1} \ar[d]^{\phi_3}
&
\mathcal{A}
\ar[r] \ar[d]^{\omega} & 
%\stackrel{e_\mathcal{A}}{\longrightarrow} 
\mathbb{Z}_\ell
\ar[r] \ar@{=}[d] &  
0 
\\
0 \ar[r] &  (\mathfrak{F}'_{s-1,k})^{\mathrm{ab}}=\mathfrak{F}_{s-1,k}'/\mathfrak{F}_{s-1,k}''
\ar[r]^-{\theta} & 
\left(\mathcal{A}_{\psi}^{\mathfrak{F}_{s-1}',\mathfrak{R}_k}\right)
%^{s-1}
\ar[r]^-{d_{1,k}} 
&
\mathcal{A}^{\mathfrak{F}_{s-1}'
,\mathfrak{R}_k}
\ar[r] & 
%\stackrel{e_\mathcal{A}}{\longrightarrow} 
\mathbb{Z}_\ell 
\ar[r] &  
0
}
\end{equation}

 The top equation is the Blanchfield-Lyndon exact sequence in eq. (\ref{BlanchfieldL}).
For the vertical arrows: $\omega$ is the map induced functorialy by the natural group homomorphism $\mathbb{Z}^{s-1} \rightarrow (\mathbb{Z}/\ell^k \mathbb{Z})^{s-1}$:
\begin{align*}
\mathcal{A}=k[[\mathbb{Z}^{s-1}]] & \stackrel{\omega}{\longrightarrow} 
\mathcal{A}^{\mathfrak{F}_{s-1}'
,\mathfrak{R}_k} =k[\mathbb{Z}/\ell^k \mathbb{Z}]
\end{align*}
The map $\phi_3$ is defined as follows:
we begin from the short exact sequence:
\[
\xymatrix{  
1\ar[r] & \mathfrak{F}_{s-1}'  \ar[r] \ar[d] &
 \mathfrak{F}_{s-1} \ar[r] 
  \ar[d]
  & \frac{\mathfrak{F}_{s-1}}{\mathfrak{F}_{s-1}'}  \ar[d] \ar[r] & 1 
\\
  1 \ar[r] & \frac{\mathfrak{F}_{s-1}' \cdot \mathfrak{R}_k}{\mathfrak{R}_k} \ar[r]  &
\frac{\mathfrak{F}_{s-1}}{\mathfrak{R}_k} \ar[r]  & 
\frac{\mathfrak{F}_{s-1}}{\mathfrak{F}'_{s-1} \cdot \mathfrak{R}_k}
 \ar[r]  &
1
}
\]
In the first row we consider the group $\mathfrak{F}_{s-1}$ as the quotient of the free pro-$\ell$ group is $s$ generators modulo the relation $x_1x_2\cdots x_s=1$. In the second row the group $\mathfrak{F}_{s-1}/\mathfrak{R}_k$ is considered as the quotient of the free group in $s$-generators modulo the relation $x_1x_2\cdots x_s=1$ and the $r$-relations generating $\mathfrak{R}_k$. 
\begin{equation}
\label{presentFree}
\xymatrix@C=15pt{
\mathcal{A} \ar[r]^-{Q_1}
&
%\left(\mathcal{A}^{\mathfrak{F}_{s-1}',\{1\}}\right)^{s}=
\mathcal{A}^s \ar[r]^-{\psi_1}  \ar[d]^{\phi_2}&
%\mathcal{A}^{\mathfrak{F}_{s-1},\{1\}}_\psi=
\mathcal{A}^{s-1} \ar[r] \ar[d]^{\phi_3} & 0
\\
\left(\mathcal{A}^{\mathfrak{F}_{s-1}',\mathfrak{R}_k}\right)^{r+1} \ar[r]^{Q_2} &
\left(\mathcal{A}^{\mathfrak{F}_{s-1}',\mathfrak{R}_k}\right)^{s} \ar[r]^-{\psi_2} &
\mathcal{A}^{\mathfrak{F}'_{s-1},\mathfrak{R}_k}_\psi \ar[r] & 0
}
\end{equation}
where $Q_1,Q_2$ are the maps appearing in proposition \ref{freeApsi}. In particular the map $Q_1$ sends 
\[
\mathcal{A} \ni \beta \mapsto \beta\cdot( 1,x_1,x_1x_2,\ldots,x_1 x_2 \cdots x_{s-1} ).
\]
The vertical map $\phi_2$  is the  reduction modulo $\Gamma$ and it is onto.
The image $\phi_3(a)$ for $a\in \mathcal{A}^{s-1}$
%= \mathcal{A}^{\mathfrak{F}'_{s-1},\{1\}}_\psi$
 is defined by selecting $b\in \mathcal{A}^{s}$ such that $\psi_1(b)=a$, and 
then $\phi_3(a)=\psi_2\circ\phi_2(b)$ as seen in the diagram bellow:
\[
\xymatrix{
  b \ar[r]^{\psi_1} \ar[d]^{\phi_2} & a \ar[d]^{\phi_3} \\
  \phi_2(b) \ar[r]^-{\psi_2} & \phi_3(a)=\psi_2 \circ \phi_2(b)
}
\]
This definition is independent from the selection of $b$.

Finally, the map $\phi$ is naturally defined 
\begin{align*}
\frac{
  \mathfrak{F}_{s-1}'}
  {\mathfrak{F}_{s-1}''
}
& \longrightarrow 
\frac{
  \mathfrak{F}_{s-1}'
\cdot \mathfrak{R}_k
  } 
  {\mathfrak{F}_{s-1}''
  \cdot  \mathfrak{R}_k
} 
\\
a \; \mathfrak{F}_{s-1}'' & \longmapsto a  \; \mathfrak{F}_{s-1}'' \cdot \mathfrak{R}_k
\end{align*}

For an explanation of these two combined sequences in terms of the ``cotangent sequence'' and a functorial point of view we refer to \cite{ParamPartC19}. 
\begin{lemma} \label{invariantRk}
The group $\mathfrak{R}_k$ is invariant under the action of $\mathrm{Gal}(\bar{\Q}/\Q)$.
\end{lemma} 
\begin{proof}
 For every generator $x_i^{\ell^k}$ and $\sigma \in \mathrm{Gal}(\bar{\Q}/\Q)$ we have 
\[
\sigma(x_i^{\ell^k})=\sigma(x_i)^{\ell^k}=
\left(
w_i(\sigma) x_i^{N(\sigma)} w_i(\sigma)^{-1}
\right)^{\ell^k}.
\]
Let $a_n$ be a sequence of integers such that $a_n\rightarrow N(\sigma)$. We have 
\[
\left(
w_i(\sigma) x_i^{a_n} w_i(\sigma)^{-1}
\right)^{\ell^k}=
\left(
w_i(\sigma) x_i^{\ell^k} w_i(\sigma^{-1})
\right)^{a_n}. 
\]
The later element is in $\mathfrak{R}_k$ since by definition 
$\mathfrak{R}_k$
is normal in $\mathfrak{F}_{s-1}$.
The limit $a_n\rightarrow N(\sigma)$ is in $\mathfrak{R}_k$ since this group is by definition closed. The result follows. 
\end{proof}

It is clear that 
$\mathcal{A}^{\mathfrak{F}_{s-1}'
,\mathfrak{R}_k}=\Z_\ell[H_0]$ can be considered through the vertical map $\omega$ as an $\mathcal{A}$-module and inherits an action of $\mathrm{Gal}(\bar{\Q}/\Q)$ by $\omega$, by 
 writing
 $\alpha\in \mathcal{A}^{\mathfrak{F}_{s-1}'
,\mathfrak{R}_k}$ as 
the
image of an element $\alpha'\in \mathcal{A}$, that is $\alpha=\omega(\alpha')$ and 
defining
\[
\sigma (\alpha)= \sigma(\omega(\alpha'))=\omega(\sigma \alpha'). 
\]
By lemma \ref{invariantRk} this action is well defined. On the other hand an element $\phi=\overline{[x_i,x_j]} \in \mathfrak{F}_{s-1,k}'/\mathfrak{F}_{s-1,k}''= \mathfrak{F}_{s-1}' / (\mathfrak{F}'_{s-1} \cap \mathfrak{F}_{s-1''} \mathfrak{R}_k )$ is sent to 
the element 
\[
\theta(\phi)=d[\bar{x}_i,\bar{x}_j]=-\bar{u}_j dx_i  +\bar{u}_i dx_j \in \mathcal{A}_\psi^{\mathfrak{F}'_{s-1},\mathfrak{R}_k}.
\] 
% \highlight[id=AK]{More interesting the quotient action as explained in Kodani p. 668}
The module $\mathcal{A}_\psi^{\mathfrak{F}_{s-1}',\mathfrak{R}_k}$ 
is 
an
 $\mathcal{A}^{\mathfrak{F}_{s-1}',\mathfrak{R}_k}$-module, described by the sequence given in eq. (\ref{Qfree-res}) and by the matrix $Q$ given in eq. (\ref{rankd2}) and is naturally acted on by the absolute Galois group. Observe also that the map $\theta$ sends the class of $[x_i,x_j]$ to $d[x_i,x_j]=u_i dx_j -u_j dx_i$, and this element is annihilated by the elements
$\Sigma_i=
\sum_{\nu=0}^{\ell^k-1} \bar{x}_i^\nu$ for $1\leq i \leq s$. 
We can see this by direct computations or by observing that in $\mathcal{A}_\psi^{\mathfrak{F}_{s-1}'
,\mathfrak{R}_k}$ we have 
\[
\Sigma_i \cdot \beta_i= \beta_{s+1} \bar{x}_1\cdots \bar{x}_{i-1}. 
\]
and the image $\theta [x_i,x_j]$ has the $s+1$ coordinate $\beta_{s+1}=0$.
The above observation generalises the definition of ideal $\mathfrak{a}_n$ in eq. (8) in the article of Ihara, \cite{Ihara1985-it}.

Therefore, 
\[
H_1( C_{\ell^k,s-1} ,\Z_\ell)\cong \theta((\mathfrak{F}'_{s-1,k})^{\mathrm{ab}}) \subset  \mathcal{A}_\psi^{\mathfrak{F}_{s-1}'
,\mathfrak{R}_k}
\]
is acted on by $\mathcal{A}^{\mathfrak{F}_{s-1}'
,\mathfrak{R}_k}/ \langle \Sigma_i: 1\leq i \leq s \rangle$, and $\mathrm{Gal}(\bar{\Q}/\Q)$ acts on it in terms of the action given in eq. (\ref{nomodule}). 
Indeed, $\mathcal{A}^{\mathfrak{F}_{s-1},\mathfrak{R}_k}_\psi$ is identified with the cokernel of the matrix $Q$, i.e. an element in $\mathcal{A}^{\mathfrak{F}_{s-1},\mathfrak{R}_k}_\psi$ is the class of an $s$-tuple which is sent to  
\[
\sigma:(\beta_1,\ldots,\beta_s)+\mathrm{Im}(Q)
\longmapsto 
(\sigma \beta_1,\ldots, \sigma \beta_s) +\mathrm{Im}(Q).
\]
This action is well defined since the space $\mathrm{Im}(Q)$ is left invariant under the action of $\mathrm{Gal}(\bar{\Q}/\Q)$. Indeed, in the commutative ring $\mathcal{A}^{\mathfrak{F}_{s-1},\mathfrak{R}_k}$, the action $\sigma(\bar{x}_i)=\bar{x}_i^{N(\sigma)}$ so 
$\sigma(\Sigma_i)=\Sigma_i$, and invariance follows by eq. (\ref{IMQeq}).

\subsubsection{On Jacobian variety of Generalized Fermat curves}
\label{JacobGenFermat}
Consider the $\ell$-adic Tate module $T(\mathrm{Jac}( C_{\ell^k,s-1} ))$ of the Jacobian of the generalized Fermat curves $ C_{\ell^k,s-1} $:
\[
T(\mathrm{Jac}( C_{\ell^k,s-1} ))
% =\mathrm{Hom}(\Q_\ell/\Z_\ell, \mathrm{Jac}( C_{k,s-1} ))(\bar{\Q})
=H_1( C_{\ell^k,s-1} ,\Z)\otimes \Z_\ell=
\frac{
	\mathfrak{F}_{s-1,k}'
}{
	\mathfrak{F}_{s-1,k}''
}.
\]
Following Ihara we consider 
\begin{equation} \label{T-def}
\mathbb{T}:=\lim_{\mycom{\leftarrow}{k}} T(\mathrm{Jac}( C_{\ell^k,s-1} ))=
\lim_{\mycom{\leftarrow}{ k}}
\frac{
	\mathfrak{F}_{s-1,k}'
}{
	\mathfrak{F}_{s-1,k}''
}
, 
\end{equation}
where the inverse limit is considered with respect to the maps
$T(\mathrm{Jac}(C_{\ell^{k+1},s-1}))\rightarrow T(\mathrm{Jac}( C_{\ell^k,s-1} ))$, 
which is induced by the map 
\[
(x_0,\ldots,x_{s-1}) \mapsto (x_0^\ell,\ldots,x_{s-1}^\ell). 
\]
Let $\bar{C}_{\ell^k,s-1}= C_{\ell^k,s-1}  \otimes_{\mathrm{Spec}\Q} \mathrm{Spec} \bar{\Q}$.
Consider also the inverse limit
\[
\lim_{\mycom{\leftarrow}{  k}} \mathrm{Gal}(\bar{C}_{\ell^k,s-1}/\mathbb{P}^1_{\bar{\Q}})=
\lim_{\mycom{\leftarrow}{k}} (\Z/\ell^k\Z)^{s-1}=\Z_\ell^{s-1}.
\]
Therefore
\[
\lim_{\mycom{\leftarrow}{ k}} \Z_\ell[\mathrm{Gal}(\bar{C}_{\ell^k,s-1}/\mathbb{P}^1_{\bar{\Q}})] \cong \mathcal{A}
\]
and $\mathbb{T}$ can be considered as an $\mathcal{A}$-module. 
Using eq. (\ref{reductionA}) we obtain
\begin{equation}
\label{prim-iso}
% \left(
\frac{
\mathfrak{F}_{s-1}'
}
{
	\mathfrak{F}_{s-1}''
}
% \right)
% ^{\mathrm{prim}}
 \cong 
 \mathbb{T}.
\end{equation}
% {\color{red} This might be wrong since the module is not free for $s\geq 3$
% Maybe we need the Primary part!}
See  \cite[sec. 13]{HyperadelicGamma} for the explicit isomorphism in the case of Fermat curves.

The geometric interpretation of this construction is that for fixed $s$-number of points we can consider all generalized Fermat curves seen as $(\Z/\ell^k \Z)^s$ ramified covers of the projective line, for $k\in \mathbb{N}$. In this way we obtain a curve $C_s$,  which is a $\Z_\ell^{s-1}$ cover of the projective line. The Burau representation and the pro-$\ell$ Burau representation can be defined in terms of such an infinite Galois cover, see \cite{MR4117575}.

This construction leads to the definition of  a subspace $\mathbb{T}^{\mathrm{prim}} \subset \mathbb{T}$ which is a free $\mathcal{A}$-module of rank $s-2$. 
% change 1 {\color{red}( s-2)?}
% \highlight[id=AK]{define Prim}. 
% \begin{remark}
Observe that the submodule of a free module is not necessarily a free module and $\mathfrak{F}_{s-1}'/\mathfrak{F}_{s-1}''$ is not necessarily free. For example in the following short exact sequence
\[
\xymatrix@C=10pt{
0 \ar[r] &  (\mathfrak{F}'_{s-1})^{\mathrm{ab}}=\mathfrak{F}_{s-1}'/\mathfrak{F}_{s-1}''
\ar[r] & 
\mathcal{A}^{s-1}
\ar[r]^{d_1} 
&
\mathcal{A}
\ar[r]  & 
%\stackrel{e_\mathcal{A}}{\longrightarrow} 
\mathbb{Z}_\ell
\ar[r]  &  
0 
}
\]
$\mathfrak{F}_{s-1}'/\mathfrak{F}_{s-1}''$ is contained in the free $\mathcal{A}$-module $\mathcal{A}^{s-1}$, but is free itself. 
The $\mathcal{A}$-module $\mathfrak{F}_{s-1}'/\mathfrak{F}_{s-1}''$ contains 
 the free module of rank $s-2$ (see \cite{Oda1985-qu},\cite[Th. 5.39]{MorishitaATIT})
\[
% (\mathfrak{F}_{s-1}'/\mathfrak{F}_{s-1}'')^{\mathrm{prim}}:=
\mathbb{T}^{\mathrm{prim}}:=
\left\{ 
(\lambda_j u_1\cdots \hat{u}_j \cdots u_{s-1})_{j=1,\ldots,s-1}: \lambda_j \in \mathcal{A}, \sum_{j=1}^{s-1} \lambda_j=0
\right\}.
\]
Set $w=u_1\cdots u_{s-1}$. Using eq. (\ref{prim-iso}) we see that a basis 
of $(\mathfrak{F}_{s-1}'/\mathfrak{F}_{s-1}'')^{\mathrm{prim}}$ is given by
\[
v_1=\left(-\frac{w}{u_1},\frac{w}{u_2},0,\ldots,0\right),\ldots,
v_{s-2}=\left(0,\ldots,0,-\frac{w}{u_{s-2}},\frac{w}{u_{s-1}}\right).
\]
In the case of Fermat curves, i.e. $s=2$ we have that $(\mathfrak{F}_{s-1}'/\mathfrak{F}_{s-1}'')^{\mathrm{prim}}=\mathfrak{F}_{s-1}'/\mathfrak{F}_{s-1}''$ and $\mathfrak{F}_{s-1}'/\mathfrak{F}_{s-1}''$ is a free $\mathcal{A}$-module, generated by $[x_1,x_2]$. Notice that the injective map $d:\mathfrak{F}_{s-1}'/\mathfrak{F}_{s-1}''
\stackrel{d}{\longrightarrow} 
\mathcal{A}^{s-1}$ is given by sending a representative
\begin{align*}
[x_i,x_j]\rightarrow d([x_i,x_j]) &= (1-x_j)dx_i- (1-x_i)dx_j
\\
&= -u_j \cdot dx_i+ u_i \cdot dx_j.
\end{align*}
% \end{remark} 
% {\color{blue}
% The 
% Set $w_{ij}=w/(u_i u_j)$. The elements 
% \[
% w_{ij}\cdot d([x_i,x_j])=w/u_j \cdot dx_j- w/u_i \cdot dx_i \in (\mathfrak{F}_{s-1}'/\mathfrak{F}_{s-1}'')^{\mathrm{prim}}.
% \]
% This means that in the image of the commutators $d([x_i,x_j])$ there is some contribution to the torsion part of $\mathfrak{F}'_{s-1}/\mathfrak{F}''_{s-1}$. 
% }

\begin{proposition} \label{44}
Let $G$ be either $\mathrm{Gal}(\bar{\Q}/\Q)$ or the braid group $B_s$.  
An element in $g\in G$ induces an action on both $\mathbb{T}$ and $\mathbb{T}^{\mathrm{prim}}$. In particular the subspace 
$\mathbb{T}^{\mathrm{prim}}$ is a free  $\mathcal{A}$-module. Thus we have a cocycle map 
\begin{align*}
\rho:  \mathrm{Gal}(\bar{\Q}/\Q)&\rightarrow \mathrm{GL}_{s-2}(\mathcal{A})\\ %change 1
\sigma & \longmapsto (a_{ij}(\sigma))
\end{align*}
This cocycle can be given in terms of the matrix
\[
\sigma(w_{ij} d[x_i,x_j])=\sum_{\nu < \mu } a_{\nu\mu}(\sigma) w_{\nu\mu} d[x_\nu,x_\mu].
\]
\end{proposition}
\begin{remark}
In \cite[sec. 5.3]{MorishitaATIT} this cocycle is identified as the Gassner representation and the relation with the classical definition in terms of Fox derivatives see \cite[chap. 3]{BirmanBraids} is studied. 
The Gassner cocycles when restricted to a certain subgroup $\mathrm{Gal}(\bar{\Q}/\Q)[1] \subset \mathrm{Gal}(\bar{\Q}/\Q)$ give rise to a representation instead of cocycle, see \cite{MorishitaATIT}.
\end{remark} 

\subsubsection{From generalized Fermat curves to cyclic covers of $\mathbb{P}^1$}
\label{Gassner2Burau}

We will now relate the Crowell sequences for the generalized Fermat curves and cyclic covers $\bar{Y}_{\ell^k}$ of the projective line as they were defined in \cite{MR4117575} using the results of \cite{ParamPartC19}. 
This will provide the relation of the Gassner representation to the Burau representation. The analogon of the Burau representation was defined in \cite[p.675]{MorishitaATIT} by reduction of the Gassner representation. Here we also consider this reduction with respect to the curve definition of the Burau representation.

We have the following diagram of ramified coverings of curves
\[
\xymatrix{
 C_{\ell^k,s-1}  
\ar[dr]^{(\Z/\ell^k \Z)^{s-1}}
&  &  
\bar{Y}_{\ell^k} \ar[dl]_{\Z/\ell^k\Z} \\
&
\mathbb{P}^{1}-\{0,1,\infty,\lambda_1,\ldots,\lambda_{s-3}\} 
&
}
\]
The passage for the corresponding representations from $ C_{\ell^k,s-1} $ to $\bar{Y}_{\ell^k}$ corresponds to the passage from the 
Gassner representation to the Burau representation, see 
\cite[prop. 3.12]{BirmanBraids} and \cite[sec. 5]{MorishitaATIT}.

% {\bf 12/2/2018}
Set $\bar{R}_{\ell^k}$ be the fundamental group of the closed curve $\bar{Y}_{\ell^k}$, which can be computed  using Schreier lemma, see \cite{MR4117575}:
\[
\bar{R}_{\ell^k}=R_{\ell^k}/\Gamma=
\left\langle
 (x_2x_1^{-1})^{x_1^{\nu}},\ldots,
(x_{s-1} x_1^{-1})^{x_1^\nu}: 0\leq \nu <\ell^k-1
\right\rangle.
\]
Let also $C_s$ be the $\Z$ cover of the projective line ramified over $s$-points. Let $R$ be its fundamental group, which by \cite{MR4117575} equals 
\[
R=
\left\langle 
\left(
x_j x_1^{-1}
\right)^{x_1^{\nu}}: \nu\in \mathbb{Z}, 2\leq j \leq s-1
\right\rangle.
\]
The fixed field of $R/ \mathfrak{R}_k$ is the function field $K_{\ell^k}$ of the curve $\bar{Y}_{\ell^k}$ and $\mathbbm{k}(C_s)$ is the function field of the curve 
$C_s$. The group $R'$ corresponds to the maximal unramified abelian extension
 $\mathbbm{k}(C_s)^{\mathrm{ur}}$ of $\mathbbm{k}(C_s)$ while $\mathfrak{R}_k$ corresponds to the maximal unramified extension $\mathbbm{k}(C_s)^{\mathrm{unrab}}$. The group $R'\cdot \mathfrak{R}_k$ corresponds to the maximal 
abelian unramified $K_{\ell^k}^{\mathrm{unrab}}$ extension of $K_{\ell^k}$. 
The groups $F_{s-1}' \cdot \mathfrak{R}_k$ and $F_{s-1}'' \cdot 
\mathfrak{R}_k$
correspond to the generalized Fermat curve $ C_{\ell^k,s-1} $ and the maximal unramified extension $ C_{\ell^k,s-1} ^{\mathrm{unrab}}$. 
The groups $F_{s-1}',F_{s-1}''$ correspond to the maximal abelian unramified extension of $K_0$ and the maximal abelian unramified
 extension of $K'$ respectively. 
\[
\xymatrix@C=5pt{
 & \mathcal{M} \ar@{-}[dr] \ar@{-}[dl] &
\\
K''  \ar@{-}[d]  \ar@{-}[dr] & &  \mathbbm{k}(C_s)^{\mathrm{ab}} \ar@{-}[d] \ar@{-}[dl] & \mathbbm{k}(C_s)^{\mathrm{unrab}}\ar@{-}[dl] \ar@{-}[d] 
\\
K' \ar@{-}[dr]&  
\mathbbm{k}(C_{\ell^k,s-1}) ^{\mathrm{unrab}}  \ar@{-}[d]
&  K_{\ell^k}^{\mathrm{unrab}}  \ar@{-}[d]|-{H_1(C_{\ell^k},\mathbb{Z})} \ar@{<-}[l] &
\mathbbm{k}(C_s)  \ar@{-}[dl]
\\
&  \mathbbm{k}(C_{\ell^k,s-1})  \ar@{-}[dr]  & K_{\ell^k}  \ar@{-}[d] \ar@{<-}[l]\\ 
&  &	K_0=\mathbbm{k}(t)
}
\;\;\;
%-------------------------
\xymatrix@C=5pt{
 & \{1\} \ar@{-}[dr] \ar@{-}[dl] &
\\
 F_{s-1}''  \ar@{-}[d]  \ar@{-}[dr] & &  \mathfrak{R}_k \ar@{-}[d] \ar@{-}[dl] & R'  \ar@{-}[dl] \ar@{-}[d] 
\\
F_{s-1}' \ar@{-}[dr]&  F_{s-1}'' \ar@{-}[d] \cdot \mathfrak{R}_k \ar@{^{(}->}[r]
& R' \cdot \mathfrak{R}_k  \ar@{-}[d]|-{H_1(C_{\ell^k},\mathbb{Z})}  &
R  \ar@/^1pc/@{-}[dl]
\\
& F_{s-1}' \cdot \mathfrak{R}_k\ar@{-}[dr] 
\ar@{^{(}->}[r] & R \cdot \mathfrak{R}_k  \ar@{-}[d] %\ar@{_{(}->}[l]
\\ 
&  &	F_{s-1}
}
\]
As in the case of generalized Fermat curves we can form the limit 
\[
\mathbb{T}_R:=\lim_{\mycom{\leftarrow}{ k}} (\mathrm{Jac}(\bar{Y}_{\ell^k}))=
\lim_{\mycom{\leftarrow}{ k}} (R_{\ell^k}/\mathfrak{R}_k)^{\mathrm{ab}}=R^{\mathrm{ab}}=H^1(\bar{Y}_{\ell^k},\Z_\ell).
\]

We now compare the Crowell sequences for the cyclic covers and the Fermat covers, following \cite{ParamPartC19}
\[
\xymatrix@C=13pt{
 &  & 0 & 0 \\
0 \ar[r] & R^{\mathrm{ab}}=\mathbb{T}_R \ar[r]  & 
\mathcal{A}_\psi^{R,\mathfrak{R}_k} \ar[r] \ar[u]  &
\mathcal{A}^{R,\mathfrak{R}_k}=\Z_\ell[[\Z/\ell^k\Z]] \ar[r]  \ar[u] &  \Z_\ell \ar[r] & 0 
\\
0 \ar[r] & (\mathfrak{F}'_{s-1})^{\mathrm{ab}}=\mathbb{T} \ar[r] \ar[u]^{\phi_1} & 
\mathcal{A}_\psi^{\mathfrak{F}'_{s-1},\mathfrak{R}_k}  \ar[r]  \ar[u]^{\phi_2} &
\mathcal{A}^{\mathfrak{F}'_{s-1},\mathfrak{R}_k}=
\Z_\ell[[(\Z/\ell^k\Z)^{s-1}]] \ar[r]  \ar[u]^{\phi_3} &  \Z_\ell \ar[r] \ar[u] &  0 \\
}
\]
The map 
$
\phi_1:\mathbb{T} \rightarrow \mathbb{T}_R
$
on Tate modules is given by the first vertical map. The action module structure is given by the commutating diagram
\[
\xymatrix{
	\mathcal{A} \times \mathbb{T} \ar[r] \ar[d] & \mathbb{T} \ar[d]^{\phi_1} \\
	\mathcal{A}^{R,\mathfrak{R}_k} \times \mathbb{T}_R \ar[r] & \mathbb{T}_R
}
\]
where the horizontal maps are the module actions and the first vertical map sends $(a,t) \mapsto (\phi_3(a),\phi_1(t))$. The map $\phi_3$ is the reduction identifying the variables $x_1,x_2,\ldots,x_{s-1}$. Let $G$ be as in proposition \ref{44}.
The map $\phi_2$ is defined in a similar way as in eq. (\ref{presentFree}).
 In particular from the reduction $\mathbb{T} \rightarrow \mathbb{T}_R$ we obtain the diagram 
\[
\xymatrix{
	G \ar[r] \ar[rd] & \mathrm{GL}_{s-2}(\mathcal{A}) \ar[d]\\
& \mathrm{GL}_{s-2}(\mathcal{A}^{R,\mathfrak{R}_k})
} %change 1  two times
\]
corresponding to the free parts of $\mathbb{T}$ and $\mathbb{T}_R$ respectively. 

% \newpage
% \setlength{\marginrulewidth}{0pt}

% \bibliography{ijmsample}
% \bibliographystyle{ijmart}

%\bibliographystyle{hplain}
\bibliographystyle{plain}

 \def\cprime{$'$}
\providecommand{\MR}{\relax\ifhmode\unskip\space\fi MR }
% \MRhref is called by the amsart/book/proc definition of \MR.
\providecommand{\MRhref}[2]{%
  \href{http://www.ams.org/mathscinet-getitem?mr=#1}{#2}
}
\providecommand{\href}[2]{#2}

\end{document}